\documentclass[reqno,15pt]{amsart}
\usepackage{amsmath}
\usepackage{amssymb}
\usepackage{amscd}
\usepackage{graphics}
\usepackage{latexsym}
\usepackage{hyperref}

\theoremstyle{plain}
\newtheorem{theorem}{Theorem}[section]
\newtheorem{thm}[theorem]{Theorem}
\newtheorem{cor}[theorem]{Corollary}
\newtheorem{lem}[theorem]{Lemma}
\newtheorem{prop}[theorem]{Proposition}

\newcounter{kludge}
\newcounter{kludgeb}
\theoremstyle{definition}
\newtheorem{defn}[theorem]{Definition}

\newtheorem{rmk}[theorem]{Remark}

\theoremstyle{remark}

\input xy
\xyoption{all}

\newcommand{\marpar}[1]{}
\newcommand{\mni}{\medskip\noindent}

\newcommand{\mbb}{\mathbb}
\newcommand{\QQ}{\mbb{Q}}

\newcommand{\ZZ}{\mbb{Z}}
\newcommand{\CC}{\mbb{C}}

\newcommand{\AAA}{\mbb{A}}
\newcommand{\PP}{\mbb{P}}

\newcommand{\mc}{\mathcal}

\newcommand{\mf}{\mathfrak}

\newcommand{\OO}{\mc{O}}

\newcommand{\wh}{\widehat}

\newcommand{\wt}{\widetilde}

\newcommand{\ol}{\overline}

\newcommand{\SP}{\text{Spec }}

\newcommand{\M}{\overline{\mc{M}}}

\newcommand{\Kgnb}[1]{\M_{#1}}

\newcommand{\Pic}[2]{\text{Pic}^{#1}_{#2}}

\newcommand{\Xx}{X}
\newcommand{\Ll}{\mathcal{L}}
\newcommand{\Cc}{C}

\newsavebox{\sembox}
\newlength{\semwidth}
\newlength{\boxwidth}

\newcommand{\Sem}[1]{%
\sbox{\sembox}{\ensuremath{#1}}%
\settowidth{\semwidth}{\usebox{\sembox}}%
\sbox{\sembox}{\ensuremath{\left[\usebox{\sembox}\right]}}%
\settowidth{\boxwidth}{\usebox{\sembox}}%
\addtolength{\boxwidth}{-\semwidth}%
\left[\hspace{-0.3\boxwidth}%
\usebox{\sembox}%
\hspace{-0.3\boxwidth}\right]%
}

\newsavebox{\semrbox}
\newlength{\semrwidth}
\newlength{\boxrwidth}

\newcommand{\Semr}[1]{%
\sbox{\semrbox}{\ensuremath{#1}}%
\settowidth{\semrwidth}{\usebox{\semrbox}}%
\sbox{\semrbox}{\ensuremath{\left(\usebox{\semrbox}\right)}}%
\settowidth{\boxrwidth}{\usebox{\semrbox}}%
\addtolength{\boxrwidth}{-\semrwidth}%
\left(\hspace{-0.3\boxrwidth}%
\usebox{\semrbox}%
\hspace{-0.3\boxrwidth}\right)%
}

\title
{Rationally simply connected varieties and pseudo algebraically closed fields} 

\author[Starr]{Jason Michael Starr}
\address{Department of Mathematics \\
  Stony Brook University \\ Stony Brook, NY 11794}
\email{jstarr@math.stonybrook.edu} 

\date{\today}

\begin{document}

%%%%%%%%%%%%%%%%%%%%%%%%%%%%%%%%%%%%%%%%%%%%%%%%%%%%%%%%%%%%%%%%%%%%
%%
%% Abstract
%%
%%%%%%%%%%%%%%%%%%%%%%%%%%%%%%%%%%%%%%%%%%%%%%%%%%%%%%%%%%%%%%%%%%%%

\begin{abstract}
  The cohomological dimension of a field is the largest degree with
  non-vanishing Galois cohomology.  Serre's ``Conjecture II'' predicts
  that for every perfect field of cohomological dimension $2$, every torsor
  over the field for a semisimple, simply connected algebraic group is
  trivial.  
  A field is perfect and ``pseudo algebraically closed'' (PAC) if
  every geometrically irreducible curve over the field has a rational
  point.  These have cohomological dimension $1$. Every 
  transcendence degree $1$ extension of such a field has cohomological
  degree $2$.  We prove Serre's ``Conjecture II'' for such fields of
  cohomological degree $2$ provided either the field is of
  characteristic $0$ or the field contains primitive roots of unity
  for all orders $n$ prime to the characteristic.  The
  method uses ``rational simple connectedness'' in an essential way.
  With the same method, we prove that such fields are $C_2$-fields,
  and we
  prove that ``Period equals Index'' for the Brauer groups of such
  fields.  Finally, we use a similar method to reprove and extend a
  theorem of Fried-Jarden: every
  perfect PAC field of positive characteristic is $C_2$.
\end{abstract}

%%%%%%%%%%%%%%%%%%%%%%%%%%%%%%%%%%%%%%%%%%%%%%%%%%%%%%%%%%%%%%%%%%%%%%
%%
%% Body
%%
%%%%%%%%%%%%%%%%%%%%%%%%%%%%%%%%%%%%%%%%%%%%%%%%%%%%%%%%%%%%%%%%%%%%%%

\maketitle

%% \tableofcontents

%%%%%%%%%%%%%%%%%%%%%%%%%%%%%%%%%%%%%%%%%%%%%%%%%%%%%%%%%%%%%%%
%%
%% Section: Introduction and Statement of Results
%% 
%%%%%%%%%%%%%%%%%%%%%%%%%%%%%%%%%%%%%%%%%%%%%%%%%%%%%%%%%%%%%%%

\section{Statement of Results} \label{sec-int}
\marpar{sec-int}

\mni
For a field $L$, the \emph{cohomological dimension} is the supremum
(possibly infinite)
over all integers $n$ such that there exists a discrete Galois module
with non-vanishing degree $n$ Galois cohomology.  For every finite
extension $L'/L$, the cohomological dimension of $L'$ is no greater
than the cohomological dimension of $L$, \cite[Proposition II.10,
p. 83]{GalCoh}. 
The cohomological dimension equals $0$ if and only if the field
is separably closed.  A \emph{Severi-Brauer variety} of dimension $n-1$ 
over $L$ is a
smooth, projective $L$-scheme $X$ such that $X\times_{\SP L}
\SP(L^{\text{sep}})$ is isomorphic to $\PP^{n-1}_{L^{\text{sep}}}$.
These are in bijection with the 
torsors over $L$ for the semisimple adjoint group
$\textbf{PGL}_n=\text{Aut}(\PP^{n-1})$, which is connected but not
simply connected,   
via the $L$-scheme of isomorphisms between $X$ and
$\PP^{n-1}_L$.   
The \emph{period}, or \emph{exponent}, 
of $X$ equals the smallest
integer $d>0$ such that there exists an invertible sheaf on $X$ whose
base change to $X\otimes_L L^{\text{sep}} \cong
\PP^{n-1}_{L^{\text{sep}}}$ 
is isomorphic to 
$\OO_{\PP^{n-1}_{L^{\text{sep}}}}(d)$.
The \emph{index} equals the smallest integer $m$ such that there
exists a closed subscheme $Y$ of $X$ whose base change in
$\PP^{n-1}_{L^{\text{sep}}}$ is a linear subvariety of
dimension $m-1$. The period divides the index, the index divides $n$,
and the period and index have the same prime factors \cite[Proposition
4.5.13]{GilleSzamuely}.  The
\emph{Period-Index Problem} asks for the smallest integer $e$
(assuming one exists), such that always the index divides the period
raised to the power $e$.

\mni
For
a perfect field $L$, the
cohomological dimension of $L$ is $\leq 1$ if and only if for every
finite separable extension $L'/L$, every Severi-Brauer variety over
$L'$ has period $1$, i.e., every $\textbf{PGL}_n$-torsor over $L'$ is  
trivial \cite[Proposition II.5, p. 78]{GalCoh}.  
For imperfect fields, typically this
condition is taken as the definition of \emph{dimension $\leq 1$}, as
opposed to ``cohomological dimension $\leq 1$''.  Serre formulated a
strong converse, ``Conjecture I'': for every field of dimension $\leq
1$, every torsor over $L$ for every semisimple and connected algebraic
group is 
trivial.  This was proved by Steinberg, \cite{Steinberg}.   
For a perfect field
$L$, by a theorem of Merkurjev-Suslin \cite[Corollary 24.9]{Suslin84}, 
the cohomological
dimension is $\leq 2$ if and only if all $\textbf{SL}_D$-torsors over $L$ 
are trivial for those semisimple
and \emph{simply connected} algebraic groups $\textbf{SL}_D$ arising
as inner forms of $\textbf{SL}_n$ over $L$.  Serre formulated a strong
converse, ``Conjecture II'':  for every perfect field of cohomological 
dimension
$\leq 2$, every torsor over $L$ for every semisimple and simply connected
algebraic group is trivial.  Serre also formulated a version of his
conjecture for imperfect fields of characteristic $p$ 
\cite[Section 5.5]{SerrePP}: 
Serre adds the hypotheses that
$[L:L^p]\leq p^2$ and that $H^3_p(F')$ is zero for all finite,
separable extensions $F'/F$.  

\mni
A field $L$ is \emph{perfect}, resp. \emph{perfect and
  pseudo-algebraically closed} (PAC), if every quasi-projective
$L$-scheme that is geometrically irreducible and zero-dimensional,
resp. one-dimensional, has an $L$-point.  
Since every perfect PAC field $L$ is
infinite, Bertini theorems imply that every quasi-projective
$L$-scheme $X$ that is geometrically irreducible contains a closed
subscheme that is geometrically
irreducible of dimension $1$ (or dimension $0$ if $X$ has dimension
$0$).  
Thus, for a perfect PAC field, every quasi-projective $L$-scheme $X$
that is geometrically irreducible has an
$L$-point.   
Via Weil's restriction
of scalars, every finite extension of a perfect PAC 
field is again a perfect PAC field.  Since Severi-Brauer varieties are
geometrically irreducible, every Severi-Brauer variety over a perfect
PAC field has a rational point, and thus it has index $1$ (so also it
has period $1$).  Therefore, every perfect PAC field has dimension
$\leq 1$.  Thus, every function field $K/L$ of transcendence degree $1$ over
a perfect PAC field $L$ has cohomological dimension $\leq 2$,
\cite[Proposition II.11, p. 83]{GalCoh}.
If $L$ has characteristic $p$, then $K$ is imperfect.  Nonetheless, 
$[K:K^p]$ equals $p$, so Serre's
modified version of ``Conjecture II'' predicts triviality for all 
$G$-torsors over $K$
for semisimple and simply connected algebraic groups $G$.  
A perfect PAC field is
\emph{nice} if either it has characteristic $0$ or if it has positive
characteristic $p$ and it contains a primitive root of unity of
order $n$ for every integer $n$ prime to $p$.  Jarden and Pop proved
the following theorem under the hypothesis that the field has
characteristic $0$ \emph{and} contains all roots of unity,
\cite{JardenPop}.  

\begin{thm} \label{thm-SerreIIpac} \marpar{thm-SerreIIpac}
For every perfect PAC field $L$ that either has characteristic zero or
contains a primitive root of unity of order $n$ for every integer $n$
prime to the characteristic $p$, for every function field
$K/L$ of transcendence degree $1$, every torsor over $K$ for every
semisimple and simply connected algebraic group is trivial.  
\end{thm}

\mni
The same method of proof also proves Period equals Index for $K$.
We thank Max Lieblich who shared with us his independent (and
different) proof of the following theorem.  This theorem can also be
proved using the Hasse principle of Efrat, \cite{Efrat}.

\begin{thm} \label{thm-PeriodIndexpac} \marpar{thm-PeriodIndexpac}
For every perfect PAC field $L$ that is nice, for every function field
$K/L$ of transcendence degree $1$, every Severi-Brauer variety over
$K$ has period equal to index.
\end{thm}

\mni
A smooth, projective, geometrically connected scheme $X$ over a field $K$
is \emph{Fano}, resp. \emph{$2$-Fano}, etc., if the first graded piece of
the Chern character of $T_{X/K} = (\Omega_{X/K})^\vee$ is positive,
resp. if the first two graded pieces are positive, etc.,
cf. \cite{dJS9}. 
A field $K$ is $C_1$, resp. $C_2$, etc., if every $K$-scheme in
$\mathbb{P}^{n-1}_K$ that is a specialization of Fano complete
intersections, resp. $2$-Fano complete intersections, has a $K$-point.   
The method proves that every function field $K$ of transcendence
degree $1$ over a nice, perfect PAC field $L$ is a $C_2$-field, first
proved by Fried-Jarden, \cite[Theorem 21.3.6]{FriedJarden}.  
In fact, there are
by now many examples of rationally simply connected varieties beyond
$2$-Fano complete intersections.  The following formulation includes
one such family discovered by Robert Findley, \cite{Findley}.

\begin{thm} \label{thm-C2nicepac} \marpar{thm-C2nicepac}
Let $L$ be a PAC field that is nice, and let $K/L$ be a function field
of transcendence degree $1$.  
Let $X$ be a $K$-scheme, and let $\mc{L}$ be an invertible sheaf on
$X$.  Then $X$ has a $K$-point in either of the following cases:
first, $X\otimes_K K^{\text{sep}}$ is isomorphic to
the common zero locus in $\PP^{n-1}_{K^{\text{sep}}}$ 
of $c$ homogeneous polynomials of degrees $(d_1,\dots,d_c)$ such that
$d_1^2 + \dots + d_c^2 < n$, and the base change of 
$\mc{L}$ is the restriction of
$\OO_{\PP^{n-1}}(1)$.  Second, there is a $K$-point if 
$X\otimes_K K^{\text{sep}}$
is isomorphic to the intersection of a degree $d$ hypersurface and
$\text{Grass}_{K^{\text{sep}}}(r,(K^{\text{sep}})^{\oplus n})$, with
its Pl\"{u}cker embedding, such that
$(3r-1)d^2 - d < n-4r-1$, and the base change of $\mc{L}$ is the
restriction of the Pl\"{u}cker $\OO(1)$.  
\end{thm}

\mni
Now let $(R,\mf{m}_R)$ be a Henselian DVR whose fraction field $K$ has
characteristic $0$ and whose residue field $L$ is a perfect PAC field
of characteristic $p$
that contains a primitive root of unity of order $n$ for every integer
$n$ prime to $p$. 
The Ax-Kochen method of Denef, \cite{DenefAxKochen}, yields the following.

\begin{thm} \label{thm-AKnicepac} \marpar{thm-AKnicepac}
The field $K$ has cohomological dimension $\leq 2$.  For Severi-Brauer
varieties over $K$, the Period equals the Index.  There exists an
integer $p_0$ such that for every such field with characteristic
$p\geq p_0$,
Serre's ``Conjecture II'' holds for $K$.  For every integer $n$ and
sequence of integers $(d_1,\dots,d_c)$ with $d_1^2 + \dots + d_c^2
\leq n-1$, there exists an integer $p_0=p_0(n;d_1,\dots,d_r)$ such that
for every field $K$ as above of characteristic $p\geq p_0$, every
closed subscheme of $\PP^{n-1}_K$ defined by equations of degrees
$(d_1,\dots,d_c)$ has a $K$-point.  Finally, for all triples of
integers $(n,r,d)$ with $(3r-1)d^2-d<n-4r-1$, there exists an integer
$p_0=p_0(n,r,d)$ such that for every $K$ as above of characteristic
$p\geq p_0$, for every polarized $K$-scheme $(\Xx_K,\Ll_K)$ whose base
change to $\ol{K}$ is a degree $d$ hypersurface in
$\text{Grass}_{\ol{K}}(r,\ol{K}^{\oplus n})$ with its Pl\"{u}cker
invertible sheaf, $\Xx_K$ has a $K$-point.
\end{thm}

\mni
Finally, a variant of the method implies a similar result for perfect
PAC fields of characteristic $p$ that are not necessarily nice.  The
$C_2$ result was first proved by Fried-Jarden, \cite[Theorem
21.3.6]{FriedJarden}, but there 
are many other rationally simply connected varieties than $2$-Fano
complete intersections.

\begin{thm} \label{thm-C2pac} \marpar{thm-C2pac}
Let $L$ be a perfect PAC field that is not necessarily nice.  
Let $X$ be an $L$-scheme, and let $\mc{L}$ be an invertible sheaf on
$X$.  Then $X$ has a $L$-point in either of the following cases:
first, $X\otimes_L \ol{L}$ together with $\mc{L}$ is isomorphic to
the common zero locus in $\mathbb{P}^{n-1}_{\ol{L}}$
of $c$ homogeneous polynomials of degrees $(d_1,\dots,d_c)$ such that
$d_1^2 + \dots + d_c^2 < n$ together with the restriction of
$\OO_{\PP^{n-1}}(1)$.  Second, there is a $K$-point if 
$X\otimes_L \ol{L}$
is isomorphic to the intersection of a degree $d$ hypersurface nd
$\text{Grass}_{\ol{L}}(r,\ol{L}^{\oplus n})$, with
its Pl\"{u}cker embedding, such that
$(3r-1)d^2 - d < n-4r-1$, and the base change of $\mc{L}$ is the
restriction of the Pl\"{u}cker $\OO(1)$.  
\end{thm}

\mni
The method uses very much the notions of \emph{rationally connected}
and \emph{rationally simply connected} varieties, together with their
specializations.  The full results are considerably stronger than the
formulations above: they also include results about arbitrary
specializations, and they give height bounds for rational points.

\mni
\textbf{Acknowledgments.}  I am very grateful to Chenyang Xu; this
article builds on the earlier joint work in \cite{StarrXu}.  
I am very grateful to Max Lieblich who explained his independent proof
that ``Period equals Index'' for function fields over PAC fields.  
I an grateful to Yi Zhu with whom I discussed the technique to
transport from characteristic $0$ to characteristic $p$ results around
Serre's ``Conjecture II'' using Nisnevich's solution of the
Grothendieck-Serre conjecture in dimension $1$.  
During the development of this article I was supported
by NSF Grants DMS-0846972 and DMS-1405709, as well as a Simons
Foundation Fellowship. 

%%%%%%%%%%%%%%%%%%%%%%%%%%%%%%%%%%%%%%%%%%%%%%%%%%%%%%%%%%%%%%%
%%
%% Section: A Bertini Theorem over Non-Algebraically Closed Fields
%% 
%%%%%%%%%%%%%%%%%%%%%%%%%%%%%%%%%%%%%%%%%%%%%%%%%%%%%%%%%%%%%%%

\section{A Bertini Theorem over Non-Algebraically Closed Fields} 
\label{sec-Bertini}   \marpar{sec-Bertini}

\mni
The following Bertini theorem extends 
\cite[Corollary 2.2]{GS} to arbitrary fields.  Let $k$ be a field.
Let $B$ be a separated,
geometrically integral $k$-scheme of dimension $m\geq 1$.  Up to
replacing $B$ by a dense, open subscheme, assume that $B$ is a normal
scheme.  
Let $u:B\to \PP^N_k$ be a
generically unramified, finite type morphism.  Let $h:B'\to B$ be a
generically finite, finite type morphism.  Let $c$ be an integer
satisfying $0\leq c \leq \max(0,m-1)$, and let $r$ denote $N-c$.  
Denote by $\text{Grass}_k(\PP^r,\PP^N)$ the Grassmannian
parameterizing linear subspaces of $\PP^N_k$ of dimension $r$, i.e.,
of codimension $c$.  This is the Hilbert scheme of $\PP^N_k$ over $k$
for the numerical polynomial $P_r(t)$ such that $P_r(d) =
\binom{d+r}{r}$ for every integer $d\geq -r$.  Denote by $\Lambda \subset
\text{Grass}_k(\PP^r,\PP^N)\times_{\SP k} \PP^N_k$ the universal linear
subspace.  
For every field extension $K/k$, for every $[L]\in
\text{Grass}_k(\PP^r,\PP^N)(\SP K)$, the associated morphism
$$
h_L: B'\times_{\PP^N_k} L\to B\times_{\PP^N_k} L,
$$
is a morphism of finite type $K$-schemes.  

\begin{defn} \label{defn-insepsec} \marpar{defn-insepsec}
An \emph{inseparable section} of $h$ is an integral 
closed subscheme $Z\subset B'$ such
that $h|_Z:Z\to B$ is dominant and the field extension $k(Z)/k(B)$ is
purely inseparable.  A \emph{rational section} is an inseparable
section such that the extension $k(Z)/k(B)$ is an isomorphism of
fields.  
The \emph{domain of definition} is the maximal
open subscheme of $B$ over which $h|_Z$ is flat and finite.
Equivalently, this is the maximal open subscheme over which $h|_Z$ is
faithfully flat. 
\end{defn}

\mni  
For every inseparable section, resp. rational section, $Z$
with domain of definition $W$,
if $B\times_{\PP^N_k} L$ 
intersects $W\times_{\SP k} \SP K$ in a dense open subset of
$B\times_{\PP^N_k} L$, 
then the associated reduced scheme of $Z\times_{\PP^N_k} L$  
is an inseparable
section, resp. rational section, of $h_L$. 

\begin{thm} \label{thm-ffBertini} \marpar{thm-ffBertini}
There exists a dense Zariski open subset $U\subset
\text{Grass}_k(\PP^r,\PP^N)$ 
and a finite, Galois extension $k'/k$ 
such that for every field extension 
$K/k$ that is linearly disjoint from $k'$,
for every $[L]\in U(\SP K)$, the restriction
map from the set of inseparable sections of $h$ to the set of
inseparable sections of $h_L$ is well-defined and is a bijection.  If
$k$ is perfect, the same holds with ``rational sections'' in place of
``inseparable sections''.
\end{thm}

\begin{proof}
If $c$ equals $0$, then $r$ equals $N$, and the result
is vacuously true with $k'=k$.
Thus, assume that $m\geq 2$, and assume that $1\leq
c \leq m-1$.  

\mni
\textbf{Restriction Map Well-Defined and Injective.}
By 
\cite[Th\'{e}or\`{e}me 4.10, 6.10]{Jou}, there exists a dense open
subscheme $U_0\subset \text{Grass}_k(\PP^r,\PP^N)$ such that for
every $K/k$ and for every $L\subset \PP^N_K$ 
with $[L]\in U_0(K)$, $B\times_{\PP^N_k} L$ is
a geometrically integral $K$-scheme.

\mni
There are only finitely many inseparable sections of $h$,
resp. rational sections of $h$.  For each, the
maximal domain of definition is a dense open subset whose complement
is a proper closed subset.  Since $u$ is generically finite, the image
of this proper closed subset in $\mathbb{P}^N_k$ is contained in a
proper closed subset.  Similarly, for any two distinct inseparable
sections, resp. rational sections, 
the intersection of the closures of the images is a closed
subset of $X$ whose closed image in $B$ does not contain the generic
point.  Thus, the image of this closed subset in $\mathbb{P}^N_k$ is
contained in a proper closed subset.  
Since $\mathbb{P}^N_k$ is irreducible, the
union of finitely many proper closed subsets is a proper closed
subset, say $C$.  The Fano scheme parameterizing linear subspaces
contained in $C$ is a proper closed subset of
$\text{Grass}_k(\PP^r,\PP^N)$.  Replace $U_0$ by the relative
complement in $U_0$ of this Fano scheme.   
Then for every $K/k$, for every $L\subset \PP^N_K$ 
with $[L]\in U_0(\SP K)$, $B\times_{\PP^N_k} L$ is a geometrically
integral $K$-scheme that intersects the domain of definition of each
inseparable section, resp. rational section, 
of $h$ and such that for any two inseparable sections, resp. rational
sections, 
$B\times_{\PP^N_k} L$ is not contained in the closure of the locus
over which the two sections are equal.  
Thus, the base change over $K$ of every
inseparable section, resp. rational section, 
of $h$ restricts to a well-defined inseparable section, resp. rational
section, 
of $h_L$, and this restriction map is injective.

\mni
\textbf{Surjectivity of Restriction Map. Noetherian Induction.}
It remains to prove that there exists a dense open subset $U$ 
of $U_0$ and
a finite Galois extension $k'/k$ 
such that for every finite $K/k$ that is linearly disjoint from $k'$, 
for every $L\subset \PP^N_K$ with $[L]\in U(\SP K)$, the
restriction map is surjective.  This is proved by
Noetherian induction for restrictions of $h$ to closed subsets of $B'$.
If $B'$ equals $C\cup Y$ for proper closed subsets $C$ and $Y$, and if
the result is proved for $C$ and $Y$, then define $k'/k$ to be the
compositum of $k'_c/k$ and $k'_Y/k$,
and define $U =U_C\cap U_Y$.  Since $B$ is
integral, resp. since $B\times_{\PP^N_k}L$ is integral, every inseparable
section, resp. rational section, has image in $C$ or in $Y$.  
Thus, the results for $C$ and $Y$
imply the result for $B'$.  Thus, assume that $B'$ is irreducible.

\mni
Similarly, if $B'$ is nonreduced, every
inseparable section, resp. rational section, over $B$, resp. over 
$B\times_{\PP^N_k}
L$, factors through the associated reduced scheme of $B'$.  Thus,
assume that $B'$ is irreducible and reduced.

\mni
\textbf{Case I. Morphism not Dominant.}
If
$h(B')$ is contained in a proper closed subset of $B$, then the same
argument as above shows that, for $d=1$ and for
$U$ a dense open subset of $U_0$, no $B\times_{\PP^N_k}L$ is contained
in $h(B')$.  Thus, the restriction map on sections is the unique set
map from the empty set to the empty set, and the result is proved.

\mni
\textbf{Case II. Morphism Birational.}
Similarly, if $h:B'\to B$ is birational, then $h_L$ is also birational
since the domain of definition of the inverse rational section
intersects the integral scheme
$B\times_{\PP^N_k} L$.  Thus, the set of sections for each
is a singleton set, and the restriction map is a bijection.

\mni
\textbf{Case III. Morphism Purely Inseparable, not Birational.} 
If $h:B'\to B$ is dominant and purely inseparable of degree $a>1$,
then $B'$ is an inseparable section.  So again, the restriction map on
inseparable sections is a bijection between singleton sets.  

\mni
For
rational sections, assume that $k$ is perfect (otherwise the argument
is much more technical).  Since $B'$ is integral and $k$ is perfect,
$B'$ is generically smooth over $k$.  Up to shrinking $B$ and $B'$,
assume that $B$ and $B'$ are $k$-smooth, and assume that $h$ is finite
and flat.
Then
$dh^\dagger:h^*\Omega_{B/k}\to \Omega_{B'/k}$ is a homomorphism of
locally free sheaves of rank $m$, and it is not surjective.  Up to
shrinking further, assume that the cokernel is locally free, so that
also the image of $dh^\dagger$ is locally free.  Thus, also the kernel
$\mc{T}^\vee_g$ of $dg^\dagger$ is locally free of positive rank
$e\geq 1$.  
The fiber product
$$
\Lambda_{U_0}' = U_0\times_{\text{Grass}_k(\PP^r,\PP^N)} \Lambda
\times_{\PP^N_k} B', 
$$
parameterizes pairs $([L],x)$ of a linear space $L$ and a point $x\in
B'$ such that $u(h(x))\in L$.  By generic flatness, up to replacing
$U_0$ by a dense open subscheme, the projection morphism
$\Lambda_{U_0,X}\to U_0$ is flat.

\mni
By \cite[Th\'{e}or\`{e}me 4.10, 6.10]{Jou},
there is a
dense open subset $V_0\subset \Lambda_{U_0}'$ parameterizing pairs
$([L],x)$ such that
$B\times_{\PP^N_k} L$ is smooth of dimension $m-c$ at $h(x)$.  For
every such $([L],x)$, the tangent space to $B\times_{\PP^N_k} L$ at
$h(x)$ gives a point in the Grassmannian bundle of $(m-c)$-dimensional
subspaces of the Zariski tangent space $T_{g(x)}B$.  The associated
morphism from $V_0$ to the Grassmannian bundle over $B$ of the tangent
bundle is dominant.  Thus, there exists a dense open $V\subset V_0$
parameterizing $([L],x)$ such that the tangent space to
$B\times_{\PP^N_k} L$ at $h(x)$ is not contained in the annihilator of
$\mc{T}_g^\vee$ (this annihilator is a subspace of the Zariski tangent
space of codimension $\geq 1$).  Since $\Lambda_{U_0}'\to U_0$ is
flat, the image in $U_0$ of $V$ is a dense Zariski open subscheme
$U\subset U_0$.  Set $d$ equal to $1$.  For every field extension
$K/k$ and for every $[L]\in U(K)$, for the generic point of
$B\times_{\PP^N_k} L$, the derivative of $h_L$ is not surjective.
Therefore $h_L$ admits no rational section.  

\mni
\textbf{Case IV. Morphism Separable, not Birational.}
Finally, assume that $B'\to B$ is dominant and the separable closure $L$
of $k(B)$ in $k(B')$ has degree $a>1$.  Denote by $B''$ the integral
closure of $B$ in $L$.  There is a factorization $B'\to B''\to B$, and
$B''\to B$ is dominant and generically \'{e}tale of degree $>1$.  Up
to shrinking $B$ and $B''$, assume that that $B$ is regular, and
assume 
that $B''\to B$ is finite and
\'{e}tale of degree $a$.
Thus, there is no inseparable section nor rational section of $B''/B$.  
The goal is to find $k'/k$ and $U$
such that for every finite extension $K/k$ that is linearly disjoint
from $k'/k$ and
for every $[L]\in U(\SP K)$, also the finite \'{e}tale morphism
$B''\times_{\PP^N_k} L\to B\times_{\PP^N_k} L$ has no rational section
(and thus no inseparable section).  Then the restriction map is again
the unique map between empty sets, which is a bijection.  Up to
replacing $B'$ by $B''$, assume that $h:B'\to B$ is finite and
\'{e}tale of degree $a>1$.

%% The kernel $\mc{T}_{g\circ f}$ of the composite,
%% $$
%% dg^\dagger \circ df^\dagger:g^*f^*\Omega_{\PP^N_k/k} \to
%% g^*\Omega_{B/k}\to \Omega_{X/k}
%% $$  
%% equals the inverse image in $g^*f^*\Omega_{\PP^N_k/k}$ of the kernel
%% of $dg^\dagger$.  Since $f$ is unramified, $\mc{K}$ is also locally
%% free of positive rank $N-m+e$.    

\mni
\textbf{Case IVa. Base Field not Separably Closed in Extension. $[k':k]>1$.}
Since $B$ is generically smooth over $k$ (being
geometrically integral), $k(B)/k$ is a separable field extension.
Thus also $k(B')/k$ is a separable field extension.  It is also a
finitely generated field extension.  Thus the algebraic closure of $k$
in $k(B')$ is a finite, separable extension $\kappa/k$.  Denote by
$k'/k$ the Galois closure of $\kappa/k$.  Note that $[\kappa:k]$ is
greater than $1$ if and only if $[k':k]$ is greater than $1$.
In this case, 
set $U$ equal to
$U_0$.  For every finite field extension $K/k$ that is linearly
disjoint from $k'/k$, and thus also linearly disjoint from $\kappa/k$,
the composite morphisms 
$$
B'\times_{\PP^N_k}L\to B\times_{\PP^N_k} L \to L \to \SP K,
$$
and
$$
B'\times_{\PP^N_k} L \to B' \to \SP \kappa,
$$
establish that, as a $K$-scheme, $B'\times_{\PP^N_k} L$ factors through
the nontrivial field extension $K\otimes_k \kappa / K$.  Finally, $K$ is
algebraically closed in $K(B\times_{\PP^N_k}L)$, since
$B\times_{\PP^N_k}L$ is geometrically integral over $K$.  Thus, there
is no rational section of $h_L$ (and thus there is no inseparable
section).  

\mni
\textbf{Case IVb. Base Field Separably Closed in Extension. $k'=k$.}
In the final case, assume that $k$ is already algebraically closed in
$X$.  Set $k'$ equal to $k$.
Now repeat the proof of \cite[Corollary 2.2]{GS}.  
The composite morphism $u\circ h:B'\to \PP^N_k$ is generically
unramified.  Thus, repeating the argument above, there exists a dense
open subset $U\subset U_0$ such that for every $K/k$ and every $[L]\in
U(K)$, $B'\times_{\PP^N_k} L$ is geometrically integral over $K$.
Finally, $h_L$ is a finite, flat morphism of degree $a>1$ 
between geometrically
integral $K$-schemes.  Thus, there is no rational section.  This
completes the proof by Noetherian induction.
\end{proof}

\begin{defn} \label{defn-PACsec} \marpar{defn-PACsec}
As above, let $B$ be a finite type scheme over a field $k$, and assume
that $B$ is separated and normal.
Let $f:X\to B$ be a finite type morphism.  A \emph{PAC section} of
$f$ is an integral closed subscheme $Y\subset E$ such that the
restriction of $f$,
$f_Y:Y\to B$, is dominant with irreducible (but possibly nonreduced) 
geometric generic fiber.  The \emph{domain of definition} is the
maximal open subscheme of $B$ over which $f_Y:Y\to B$ is faithfully flat.
\end{defn}

\begin{cor} \label{cor-ffBertini} \marpar{cor-ffBertini}
For every finite type morphism $f:X\to B$, if $f$ has no PAC
section, then there exists a finite Galois extension $k'/k$ 
and a dense open
subset $U\subset \text{Grass}_k(\PP^r,\PP^N_k)$ such that for every
field extension $K/k$ that is linearly disjoint from $k'/k$,
for every
$[L]\in U(\SP K)$, the restriction $f_L:X\times_{\PP^N_k}L\to
B\times_{\PP^N_k} L$ also has no PAC section.
\end{cor}

\begin{proof}
There exists a finite type, surjective monomorphism $i:X'\to X$ such
that every 
connected component $X'_i$ of $X'$ is regular and separated.  Every
PAC section of $X'\to B$ maps under $i$ to a PAC section of $X\to B$.
Conversely, for every PAC section $Y$ of $X\to B$, the generic point
$\eta_Y$ lifts uniquely to $X'$, and the closure of this generic point
in $X'$ gives a PAC section of $X'\to B$.  The same argument holds for
$X'\times_{\PP^N_k} L \to X\times_{\PP^N_k} L \to B\times_{\PP^N_k}
L$.  Thus, up to replacing $X$ by $X'$, assume that every connected
component $X_i$ of $X$ is regular and separated, and assume that 
each restriction morphism
$f|_{X_i}:X_i\to B$ has no PAC section.

\mni  
Up to shrinking $B$,
assume that every $f|_{X_i}:X_i\to B$ is flat, and that $B$ is regular.  
The separable closure of
$k(B)$ in $k(X_i)$ is a finite, separable extension of $B$.  Denote
by $h_i:B_i\to B$ the integral closure of $B$ in this finite, separable
extension.  
Since
$B$ is finite type over a field, $B$ is excellent.  Thus, $h_i$
is finite.  By construction, $h_i$ is generically \'{e}tale.  Up to
shrinking $B$ further, assume that $h_i$ is everywhere finite and
\'{e}tale.  

\mni
Since $X_i$ is regular, it is normal.  Thus, $f|_{X_i}$
factors through $h_i$, i.e., there exists a finite type, dominant
morphism $g_i:X_i\to B_i$ such that $f_{X_i}$ equals $h_i\circ g_i$.    
By
construction, the geometric generic fiber of $g_i$ is
irreducible.  By \cite[Th\'{e}or\`{e}me 4.10, 6.10]{Jou},
there exists a dense open
subscheme of $B_i$ over which $X_i$ is faithfully flat with
irreducible geometric fibers.  Since $B_i$ is finite over $B$, up to
shrinking $B$ further, assume that this dense open subscheme equals
all of $B_i$.  Thus, every PAC section of $f|_{X_i}$ maps under $g_i$ to
a PAC section of $h_i$.  Conversely, for every PAC section $Z$ of
$h_i$, the inverse image $g_i^{-1}(Z)$ is a PAC section of $f|_{X_i}$.
In particular, since $f|_{X_i}$ has no PAC section, also $h_i$ has no PAC
section.  Since $h_i$ is finite, PAC sections are the
same as inseparable sections.  So $h_i$ has no inseparable sections.  Denote
by $h:B'\to B$ the disjoint union of the finitely many morphisms $h_i$.

\mni
By Theorem \ref{thm-ffBertini}, there exists a finite Galois extension
$k'/k$ and a
dense open subset $U\subset \text{Grass}_k(\PP^r,\PP^N)$ such that for
every field extension $K/k$ that is linearly disjoint from $k'/k$,
for every $[L]\in U(\SP K)$, also $h_L$ has no inseparable section.  Every
PAC section of $f_L:X\times_{\PP^N_k} L \to B\times_{\PP^N_k}L$ maps under
$g$ to a PAC section of $h_L$.  Since $h_L$ is finite and
has no inseparable section, it has no PAC section.  Thus, $f_L$ has no
PAC section.
\end{proof}

%%%%%%%%%%%%%%%%%%%%%%%%%%%%%%%%%%%%%%%%%%%%%%%%%%%%%%%%%%%%%%%
%%
%% Section: Fields that Admit Rational Points on Specializations of
%% Rationally Connected Varieties  
%% 
%%%%%%%%%%%%%%%%%%%%%%%%%%%%%%%%%%%%%%%%%%%%%%%%%%%%%%%%%%%%%%%

\section{Fields that Admit Rational Points on Specializations of
Rationally Connected Varieties} \label{sec-RCsolv} 
\marpar{sec-RCsolv}

\mni
After
James Ax introduced PAC fields,
he asked whether every specialization of Fano hypersurfaces in
$\mathbb{P}^{n-1}$  over
a perfect PAC field $L$ has an $L$-point \cite[Problem 3]{Ax}.  A 
projective variety over a field $F$ is
\emph{rationally connected}, resp. \emph{separably rationally
  connected}, 
if for every algebraically closed field
extension $E/F$, resp. for every separably closed field extension
$E/F$, 
every pair of $E$-points of the variety is
contained in the image of an $E$-morphism from $\mathbb{P}^1_{E}$.  In
characteristic $0$, these two definitions agree.
Sufficiently general Fano
hypersurfaces are \emph{separably rationally connected
  varieties}, \cite{KMM} (in characteristic $0$), \cite{Zhu2} (in
arbitrary characteristic).  Thus, Ax was asking about rational points
on specializations of certain separably rationally connected varieties.

\mni
The most common separably rationally connected varieties
have unobstructed deformations, e.g., all Fano manifolds in
characteristic $0$, (standard) projective homogeneous spaces in all
characteristics, 
Fano complete intersections of ample divisors in 
(standard) projective homogeneous
spaces in all characteristics, etc.
However, to formulate results that
also apply to those rationally connected varieties with obstructed
deformations, it is necessary to address \emph{ramification},
particularly in mixed characteristic.  For DVRs
$(\Lambda,\mf{m}_{\Lambda})$ and $(R,\mf{m}_R)$, a local homomorphism
$\phi:\Lambda \to R$ is \emph{regular} if 
\begin{enumerate}
\item[(i)] $\phi(\mf{m}_{\Lambda})R$ equals $\mf{m}_R$, i.e., $\phi$
  is \emph{weakly unramified} (note that this implies that $\phi$ is
  injective), 
\item[(ii)] the residue field extension $\Lambda/\mf{m}_\Lambda \to
  R/\mf{m}_R$ is separable (note that this holds automatically if
  $\Lambda/\mf{m}_\Lambda$ is perfect), and 
\item[(iii)] the fraction field extension is separable (note that this
  holds automatically if the fraction field has characteristic $0$).
\end{enumerate}
If the local homomorphism is essentially of finite type and if
$\Lambda$ is complete, then the first two hypotheses imply the third.
The importance of regularity here has to do with \emph{smooth
  parameter spaces}.

\begin{defn} \label{defn-param} \marpar{defn-param}
For an integral scheme $S$, a \emph{parameter space} over $S$ is a triple
$(M \to S, f_M:\Xx_M\to M,\Ll)$ of a smooth $S$-scheme $M$ of
pure relative dimension $m$, a flat, projective morphism $f_M$, and an
invertible sheaf $\Ll$ on the fiber of $\Xx_M$ over $\SP
\text{Frac}(S)$.  
\end{defn}

\begin{lem} \label{lem-ft} \marpar{lem-ft}
Let $(R,\mf{m}_R)$ be a DVR, let
$\Xx_R$ be a flat, projective $R$-scheme,
and let $\Ll_{\text{Frac}(R)}$ be an invertible sheaf on
$\Xx_{\text{Frac}(R)}$.  
Let $(\Lambda,\mf{m}_\Lambda)\to (R,\mf{m}_R)$ be a local homomorphism
of DVRs that is regular.  
There
exists a parameter space over $S=\SP \Lambda$, and there exists a dominant
$S$-morphism $\zeta:\SP R \to M$ such that $(\Xx_R,\Ll_R)$ is the
pullback by  $\zeta$ of $(\Xx_M,\Ll)$. 
\end{lem}

\begin{proof}
For any subring $L$ of $R$,
by the usual limit arguments, there exists a subring
$A\subset R$ containing $L$ such that $L\to A$ is finitely
generated, there exists $\Xx_A\to A$ a flat, projective scheme, and
there exists an invertible sheaf on $\Xx_{\text{Frac}(A)}$.  Without
loss of generality, also assume that $A$ contains a generator for the
principal ideal $\mf{m}_R \subset R$.
\mni
Now, set $L$ equal to $\phi(\Lambda)$.
Since
$\text{Frac}(A)$ is a subextension of $\text{Frac}(\Lambda)\to
\text{Frac}(R)$, which is a separable extension by (iii), also
$\text{Frac}(\Lambda)\to \text{Frac}(A)$ is separably generated.
Since $A\otimes_\Lambda \text{Frac}(\Lambda)$ is finitely generated
over $\text{Frac}(\Lambda)$ with separably generated fraction field, 
there exists $a\in 
A\setminus\{0\}$ such that $A[1/a]\otimes_\Lambda
\text{Frac}(\Lambda)$ is a smooth algebra over $\text{Frac}(\Lambda)$.
Every uniformizing element $\pi$ of $\Lambda$ is also a uniformizing
element of $R$ by (i).  Thus, for $a\in A\subset R$, there exists an
integer $e\geq 0$ and there exists $u\in R\setminus \mf{m}_R$ such
that $a$ equals $u\pi^e$.  Adjoining $u$ to $A$ does not change
$A\otimes_\Lambda \text{Frac}(\Lambda)$.  Thus, assume that $u$ is in
$A$.  Then $A[1/u]\otimes_\Lambda \text{Frac}(\Lambda)$ equals
$A[1/a]\otimes_\Lambda \text{Frac}(\Lambda)$.  Since $u$ is in
$R\setminus \mf{m}_R$, also $1/u$ is in $R\setminus \mf{m}_R$.  Thus,
adjoin $1/u$ to $A$, and assume that $A\otimes_\Lambda
\text{Frac}(\Lambda)$ is smooth over $\text{Frac}(\Lambda)$.

\mni
Since $A$ is a finite type, flat $\Lambda$-algebra such that
$A\otimes_\Lambda \text{Frac}(\Lambda)$ is smooth over
$\text{Frac}(\Lambda)$, and by hypotheses (i) and (ii), there exists a
N\'{e}ron desingularization, \cite[Tag 0BJ6]{stacks-project}. 
Precisely, there exist finitely many ``N\'{e}ron blowups'', $A\mapsto
A[\mf{p}/\pi]$ where $\mf{p}$ equals $\mf{m}_R\cap A$, after which
$\Lambda \to A$ is smooth at $\mf{p}$, i.e., $A/\mf{m}_AA$ is smooth
over $\Lambda/\mf{m}_A$ at the prime $\mf{p}/\mf{m}_AA$.  Thus, there
exists $v\in A\setminus \mf{p}$ such that $A[1/v]$ is smooth over
$\Lambda$.  Since $v$ is in $R\setminus \mf{m}_R$, $1/v$ is also in
$R\setminus \mf{m}_R$.  Thus, $A[1/v]$ is a subring of $R$.  So after
replacing $A$ by $A[1/v]$, now $A$ is a subring of $R$ that is a
finitely generated $\Lambda$-algebra that is smooth.  Define $M$ to be
$\SP A$.   
\end{proof}

\begin{defn} \label{defn-primereg} \marpar{defn-primereg}
A \emph{prime finite}
DVR $(\Lambda,\mf{m}_\Lambda)$
is a DVR
whose residue field $\Lambda/\mf{m}_\Lambda$ is a finite extension of
the prime subfield, i.e., either the residue field is  
a finite field if the characteristic
is positive, or it is a number field if the characteristic is $0$.  A 
DVR $(R,\mf{m}_R)$ is
\emph{prime regular}, or
\emph{regular over a DVR whose residue field is finite over
  the prime subfield}, if there
exists a prime finite DVR $(\Lambda,\mf{m}_\Lambda)$ 
and a local homomorphism $\phi:\Lambda\to R$ that is regular.  
\end{defn}

\begin{lem} \label{lem-equi} \marpar{lem-equi}
Every equicharacteristic DVR is prime regular.
\end{lem}

\begin{proof}
Let $(R,\mf{m}_R)$ be an equicharacteristic DVR.  Denote by $F\subset
R$ the prime subfield.  Let $\theta\in \mf{m}_R$ be a generator.
Since $\theta$ is not an invertible element of $R$, $\theta$ is not in
$F$.  
In fact,
$\theta$ is transcendental over $F$, for otherwise a minimal
polynomial $m_\theta(t) = t^d + \dots + a_1\theta+ a_0$ has
degree $d\geq 1$ and gives a relation $1=
-a_0^{-1}\theta(a_1+\dots+\theta^{d-1})$.  This implies that $\theta$ is
invertible in $R$.  Thus, $F[\theta]$ is a copy of the polynomial ring
in $R$.  Since the multiplicative system $F[\theta]\setminus \theta
F[\theta]$ is contained in $R\setminus \mf{m}_R$, $R$ contains the
ring of fractions $\Lambda = F[\theta]_{\langle \theta \rangle}$.  The
inclusion of DVRs $(\Lambda,\mf{m}_\Lambda)\to (R,\mf{m}_R)$ is a
local homomorphism.  It is weakly unramified since $\theta\in
\mf{m}_\Lambda$ is a generator of $\mf{m}_R$.  Since the
residue field $F$ of $\Lambda$ is perfect, every field extension of
the residue field is separable.  Thus, it only remains to check that
$K=\text{Frac}(R)$ is separable over $F(\theta)$.

\mni
In characteristic $0$ this is automatic.  Assume the characteristic
equals $p$.  Since $F(\theta)^{1/p}$ equals $F(\theta)[t]/\langle
t^p-\theta \rangle$, we need to prove that $A=K[t]/\langle t^p-\theta
\rangle$ contains no nonzero nilpotent $\alpha$ with $\alpha^p$ equal to
$0$.  The $K$-vector space $A$ is free with basis
$(1,t,\dots,t^{p-1})$.  So every element $\alpha$ of $A$ has a unique
decomposition, 
$$
\alpha = a_0 + a_1t + \dots + a_{p-1}t^{p-1}.
$$
%Since $t^p$ equals $\theta$, which is
%invertible, if $(t^r\beta)^p$ equals $0$, then $\beta^p$ equals $0$.
%Thus, if there exists a nilpotent element $\alpha$, then there exists
%such an $\alpha$ with $a_0$ nonzero.  

\mni
Let $\alpha$ be a nonzero element, i.e., some $a_i$ is nonzero.  
Denote by $e\in \ZZ$ the minimum
of the valuations of those $a_i\in K$ that are nonzero.  Then up to
replacing $\alpha$ by 
$\theta^{-e}\alpha$, assume that every $a_i$ is in $R$, and at least
one $a_i$ has valuation $0$.  Let $\ell$ be the minimal $i$ with
$0\leq i \leq p-1$ such that $a_i$ has valuation $0$.  Then
$a_\ell$ is invertible, so that also $a_\ell^p$ is invertible.
Therefore $a_\ell^p\theta^\ell$ has valuation $\ell$.  

\mni
On the other hand, for every $m$ with $0\leq m\leq p-1$ and $m\neq
\ell$, either $a_m$ equals $0$ so that $a_m^p\theta^m$ equals $0$, or
$\text{val}(a_m)>0$ so that $a_m^p\theta^m$ has valuation $\geq p >
\ell$, or $\text{val}(a_m)$ equals $0$ but $m>\ell$ so that again
$a_m^p\theta^m$ has valuation $m>\ell$.  So also the sum
$$
\sum_{0\leq m \leq p-1, m\neq \ell} a_m^p\theta^m,
$$ 
is either zero or has valuation $\geq \ell+1$.  Thus the full sum,
$$
\alpha^p = a_\ell^p \theta^\ell + \sum_{0\leq m\leq p-1,m\neq \ell}
a_m^p \theta^m
$$
is nonzero of valuation $\ell$.  Therefore, $A$ contains no nonzero
nilpotent elements, and $K$ is separable over $F(\theta)$.
\end{proof}

\mni
Because of the lemma, the only DVRs that are not prime regular are
mixed characteristic DVRs that are not weakly unramified over a Cohen
ring, e.g., for $e>1$, the localization of $\ZZ[x,y]/\langle y^e-px \rangle$ at
the height one prime generated by $p$ and $y$.  

\mni
For a DVR $(\Lambda,\mf{m}_\Lambda)$, for a regular local
homomorphism $(\Lambda,\mf{m}_\Lambda) \to (R,\mf{m}_R)$, 
for every pair $(\Xx_R\to \SP R,\Ll_{\text{Frac}(R)})$ of a flat,
projective $R$-scheme $\Xx_R$ and an invertible sheaf on the generic
fiber, there exists a parameter space $(M\to \SP \Lambda,f_M:\Xx_M\to
M,\Ll)$ and a $\Lambda$-morphism $\zeta:\SP R\to M$ pulling back
$(\Xx_M,\Ll)$ to $(\Xx_R,\Ll_{\text{Frac}(R)})$,  
by
Lemma \ref{lem-ft}.  Thus every extension field of $R/\mf{m}_R$ admits
a morphism to $M_0=M\times_{\SP \Lambda} \SP
(\Lambda/\mf{m}_\Lambda)$, and this morphism is even dominant.

\mni
For a DVR $(\Lambda,\mf{m}_\Lambda)$ and for a parameter
space as above, 
let $F$ be a field (not necessarily finite), 
let $E$ be the function field of a
geometrically integral $F$-scheme of dimension $d$, 
and let $z:\SP E\to
M_0$ be a morphism, not necessarily dominant.  
Let $M^o_{\eta}$ be a specified dense open subset of the
generic fiber $M_\eta=M\times_{\SP \Lambda} \SP \text{Frac}(\Lambda)$
(for instance the entire generic fiber, but there are applications
when $M^o_{\eta}$ is a smaller dense open).

\begin{defn} \label{defn-int} \marpar{defn-int}
An \emph{integral extension} of $z$ is a
triple 
$$
(B\to \SP \Lambda, \pi^o_B:C^o_B\to B,z_B:C^o_B\to M),
$$
and a pair
$$
(\psi_B:\text{Frac}(B_0)\to F,\psi_E:F(C^o_F)\to E)
$$ 
consisting of a smooth,
quasi-projective, surjective morphism $B\to \SP \Lambda$ with 
integral closed fiber and generic fiber,
a 
smooth, quasi-projective, surjective morphism $\pi^o_B$ of relative
dimension $d$ with geometrically integral fibers, 
and a $\Lambda$-morphism $z_B$ together with 
a field 
homomorphism $\psi_B:\text{Frac}(B_0)\to F$ for the fraction field field
of the integral scheme $B_0=B\otimes_\Lambda
\Lambda/\mf{m}_\Lambda$ and 
an isomorphism of $F$-extensions 
$\psi_C:
F(C^o_F) \to E$
for the fraction field of the integral scheme
$C^o_F=C^o_B\times_B \SP F$ 
such that
\begin{enumerate}
\item[(i)] $z_B^{-1}(M^o)$ contains the generic fiber $C^o_B\times_{\SP \Lambda}\SP
  \text{Frac}(\Lambda)$ 
\item[(ii)] the morphism $z$ equals the
composition of $\SP \psi:\SP E \to \SP F(C_F)$ and the morphism $\SP
F(C^o_F) \to M$ induced by $z_B$.  
\end{enumerate}
\end{defn}

\begin{rmk} \label{rmk-integralreg} \marpar{rmk-integralreg}
For an integral extension, the stalk $R$ of the structure sheaf of $B$
at the generic point of the closed fiber $B_0=B\times_{\SP \Lambda}
\SP (\Lambda/\mf{m}_\Lambda)$ is a DVR that is regular over
$\Lambda$ (and also essentially of finite type) since $B$ is smooth
over $R$.  
\end{rmk}

\begin{lem} \label{lem-extend} \marpar{lem-extend}
For a prime finite DVR $(\Lambda,\mf{m}_\Lambda)$, for a parameter
space over $\Lambda$ and a specified dense open subset $M^o_\eta$ of
the generic fiber, for every pair $(E/F,z:\SP E\to M)$ as above,
there exists an integral extension. 
\end{lem}

\begin{proof}
Notice first, since $F$ is algebraically closed in $E$, the algebraic
closure in $E$ of the prime subfield is actually a subfield of $F$.
Thus, since $\Lambda/\mf{m}_\Lambda$ is a finite extension of the
prime field, the subfield $\Lambda/\mf{m}_\Lambda$ of $E$ is actually
a subfield of $F$.
By limit
arguments there exists a subring $A_0\subset F$ and a
smooth morphism $\pi^o_{B_0}:C^o_{B_0}\to  \SP A_0$ such that $E$
equals $F(C^o_F)$ and such that $\Lambda/\mf{m}_\Lambda \to A_0$ is
finite type.  Up to adjoining finitely many elements of $F$ to $A_0$,
up to replacing $C^o_{B_0}$ by the base change over this larger ring $A_0$,
and up to replacing $C^o_{B_0}$ by a dense Zariski open subscheme,
there also exists a $\Lambda$-morphism $z_{B_0}:C^o_{B_0}\to M$ that
induces $z$.  Also, since $\Lambda/\mf{m}_\Lambda$ is perfect, up to
inverting one nonzero element of $A_0$, the scheme $B_0=\SP A_0$ is
smooth over $\Lambda/\mf{m}_\Lambda$.

\mni
It remains to extend from $\Lambda/\mf{m}_\Lambda$ to all of $\Lambda$ the
triple $(B_0=\SP A_0$, $\pi^o_{B_0}:C^o_{B_0}\to
B_0,z_{B_0}:C^o_{B_0}\to M)$.  This follows by the method of
\cite[Section 3]{StarrXu}.  First, since $A_0$ is a smooth,
finite type algebra over $\Lambda/\mf{m}_\Lambda$, up to a further
localization, it is the quotient of a polynomial ring over
$\Lambda/\mf{m}_\Lambda$ by an ideal generated by a regular sequence.  
Lifting the coefficients of the polynomials in this regular sequence,
there exists a $\Lambda$-smooth algebra $A'$ with $A'\otimes_\Lambda
\Lambda/\mf{m}_\Lambda$ equal to $A_0$.  If $d$ equals $0$, define
$C'\to \SP A'$ to be the identity.  For $d\geq 1$, up to localizing
$A_0$ and replacing
$C^o_{B_0}$ by a dense Zariski open, realize $C^o_{B_0}$ as a dense
open subset of a hypersurface in $\PP^{d+1}_{A_0}$.  For a general
lift to $A'$ of the coefficients of the defining polynomial of this
hypersurface, the lift of the hypersurface in $\PP^{d+1}_{A'}$ is flat
over a dense open subset of $\SP A'$ that contains $\SP A_0$, and it
is smooth over a dense open subset.  Up to localizing $A_0$ and $A'$
further, this hypersurface is $A'$-flat.  Since the geometric generic
point is a smooth hypersurface of dimension $d\geq 1$ in projective
space, it is integral.  Define $C'$ to be a dense
open subset that intersects the fiber over $\SP A_0$ and that is
smooth over $\SP A'$.  Since the geometric generic fiber is integral,
up to shrinking $C'$ further, $C'$ has geometrically integral fibers
over $\SP A'$. 

\mni
Now consider the graph of $z_{B_0}$ as a closed subscheme of the fiber
product $C'\times_{\SP \Lambda} M$.  Up to shrinking further, it is an
irreducible component (of multiplicity $1$) of a complete intersection
of ample divisors in the closed fiber of $C'\times_{\SP \Lambda} M$.
As above, lift the coefficients of the defining equations to lift
$z_{B_0}$ to a closed subscheme
curve $C$ of $C'\times_{\SP \Lambda} M$ whose generic fiber over
$\text{Frac}(\lambda)$ is a general complete intersection of ample
divisors.  Since $\text{Frac}(\Lambda)$ is an infinite field, for an
appropriate choice of the lifts of the coefficients, $C$ is smooth and
the image in $M_\eta$ intersects $M^o_\eta$.

\mni
There is an issue about irreducibility of geometric fibers.
Choosing projective models over $\Lambda$ of all of the schemes,
the Stein factorization $\ol{B}$ of $C\to \SP A'$ (roughly
the integral closure of $A'$ in the fraction field of $C$) may be
nontrivial.  However, since the closed fiber $C_0$ is smooth over
$A_0$ with geometrically integral fibers, $\SP A_0$ is an irreducible
component of the closed fiber $\ol{B}_0$.  
Replace $\ol{B}$ by the open complement in
$\ol{B}$ of the union of the finitely many irreducible components of
$\ol{B}_0$ different from $\SP A_0$.  The restriction of $C$ over
$\ol{B}$ now has geometrically irreducible fibers.  Define $C^o_B$ to be the
open subset of $C$ that is the smooth locus of the morphism to
$\ol{B}$. 
Define $B$ to be the open image in $\ol{B}$ of this smooth morphism.
Define $z_B$ to be the restriction to the locally closed subscheme
$C^o_B$ of $C\times_{\SP \Lambda} M$ of the projection to $M$.  By
construction, the inverse image of $M^0_\eta$ in the generic fiber
$C^o_\eta = C^o_B\otimes_{\SP \Lambda}\SP \text{Frac}(\Lambda)$ is a
dense open.  The complement is a proper closed subset $D_\eta$.
Since $C^o_B$ is flat over $\SP \Lambda$, the closure $D$ of $D_\eta$
is flat over $\SP \Lambda$.  Thus, $D_\eta$ cannot contain the
irreducible component $C^o_{B_0}$.  After replacing $C^o_B$ by the open
complement of $D_\eta$, 
this
triple $(B\to \SP \Lambda,\pi^o_B:C^o_B\to B,z_B:C^o_B\to M)$ is an
integral extension of $z$. 
\end{proof}

\mni
Here is a precise formulation of 
existence of rational points for specializations of
separably rationally connected varieties.

\begin{defn} \label{defn-RCfld} \marpar{defn-RCfld}
A field $L$ is \emph{RC solving}, or \emph{admits rational points on
  specializations of separably rationally connected varieties},  
if for every projective, flat
scheme $\Xx_R$ over a prime regular
DVR $R$ 
such that $\Xx_R\times_{\SP R} \SP \ol{\text{Frac}(R)}$ is
smooth, integral, and separably rationally connected, 
for every field
extension $z^*:R/\mf{m}_R \hookrightarrow L$, $\Xx_R\times_{\SP R} \SP
L$ has an
$L$-rational point.  The field $L$ is \emph{characteristic $0$ RC
  solving} if the condition holds for every prime regular DVR $R$
whose fraction field has characteristic $0$.
\end{defn}

\begin{rmk} \label{rmk-reminder} \marpar{rmk-reminder}
To summarize the lemmas above, the prime regular hypothesis is
automatic if $(R,\mf{m}_R)$ is an equicharacteristic DVR.  Also, $L$
is RC solving if and only if, for every prime finite DVR
$(\Lambda,\mf{m}_\Lambda)$, for every parameter space over $\SP
\Lambda$ such that there is a dense open subset $M^o_\eta$ of the
generic fiber over which $f_M$ is smooth with geometric fibers that
are integral and separably rationally connected, for every morphism to
the closed fiber
$\SP L \to M_0$, the pullback $\Xx_M\times_M \SP L$ has an $L$-point.
\end{rmk}

\begin{thm}\cite[Lemma 2.5]{GHMS} \label{thm-RCfib}
  \marpar{thm-RCfib}
For every algebraically closed field $k$, every function field $k(C)$ of a
geometrically integral, smooth, projective $k$-curve $C$ is RC solving.
\end{thm}

\begin{proof}
The notation is as in the definition.  By hypothesis there 
exists a prime finite DVR
$(\Lambda,\mf{m}_\Lambda)$ and a local homomorphism
$(\Lambda,\mf{m}_\Lambda) \to (R,\mf{m}_R)$.  
By Lemma \ref{lem-ft}, there exists a parameter space $(M\to
S,f_M:\Xx_M\to M, \Ll)$ and a $\Lambda$-morphism $\zeta:\SP R\to M$ pulling
back $\Xx_M$ to $\Xx_R$.  In particular, the geometric generic fiber
$\Xx_R\otimes_R \ol{\text{Frac}(R)}$ is the base change of the
geometric generic fiber of $f_M$.  Each of the following properties of
a proper scheme over an algebraically closed field holds if and
only if it holds after base change to an arbitrary algebraically
closed field extension: smoothness, integrality, separable rational
connectedness.  Thus, these properties all hold for the geometric
generic fiber.  Thus the generic point of $M$ is contained in the 
maximal open subscheme $M^o_\eta$ of the
generic fiber $M_\eta = M\otimes_\Lambda \text{Frac}(\Lambda)$ over
which the geometric fibers of $f_M$ are smooth, integral, and
separably rationally connected. So $M^o_\eta$ is a dense open subset
of the generic fiber $M_\eta$.

\mni
Define $z:\SP k(C) \to M$ to be the composition of $\SP k(C) \to \SP
R/\mf{m}_R$ and the restriction to $\SP R/\mf{m}_R$ of $\zeta$.  
By Lemma \ref{lem-extend}, there exists an integral extension of $z$.
Denote by $f_C:\Xx_C\to C^o_B$ the pullback of $\Xx_M$ by $z_B$.  By
construction, $\pi^o_B:C^o_B\to B$ is smooth, quasi-projective of
relative dimension $1$, and the geometric generic fiber of $f_C$ is a
smooth, integral, separably rationally connected variety.  Thus, by
\cite{GHS}, \cite{dJS}, 
there exists a finite extension $K'$ of the
fraction field of $B$ and a $C^o_B$-morphism $s:\SP K'\times_B
C^o_B\to \Xx_C$.  Consider the DVR 
$(\OO,\mf{n})$ that is the stalk 
$\OO_{B,\eta_{B_0}}$ of the structure sheaf of $B$ at the generic
point of the closed fiber $B_0 = B\times_{\SP
  \Lambda}\SP(\Lambda/\mf{m}_\Lambda)$.  By the Krull-Akizuki theorem,
there exists a DVR $(\OO',\mf{n}')$ 
with fraction field $K'$ that dominates
$(\OO,\mf{n})$ (but $\OO'$ is not necessarily a finite $\OO$-module).  
Since $f_C$ is proper, by the valuative criterion of
properness, the maximal domain $V$ of definition of the rational
transformation $s:\SP \OO' \times_B C^o_B\to \Xx_C$ intersects the
closed fiber $\SP(\OO'/\mf{n}') \times_B C^o_B\to \Xx_C$.  Thus, $s$
gives a rational section of $f_C$ over
the base change of $C^o_{B_0}$ to the extension field $\OO'/\mf{n}'$
of the fraction field $\text{Frac}(B_0)$.  So the same holds after
extension 
to any bigger field, e.g., algebraic closure of $\OO'/\mf{n}'$.  

\mni
Existence of
a section of a morphism of finite type schemes over an algebraically
closed field is a property that holds if and only it holds after an
arbitrary extension to an algebraically closed field.  Thus, since it
holds after extension from $\text{Frac}(B_0)$ to the algebraic closure
of $\OO'/\mf{n}'$, it also holds after extension from
$\text{Frac}(B_0)$ to $k$.  Therefore there is a $k(C)$-point of
$\Xx_R\otimes_R k(C)$.  
\end{proof} 

\begin{thm}\cite{Esnaultpadic}, \cite{EsnaultXu}
\label{thm-Esnault} \marpar{thm-Esnault}
Every finite field is RC solving.
\end{thm}

\begin{rmk} \label{rmk-Esnault} \marpar{rmk-Esnault}
Every subfield of a finite field is a finite field.  So for a finite field
$L$, for every DVR $(R,\mf{m}_R)$ and field homomorphism
$R/\mf{m}_R\to L$, already $(R,\mf{m}_R)$ is prime finite.  So
$(R,\mf{m}_R)$ is prime regular.
\end{rmk}

\begin{proof}
When $R$ has mixed characteristic, this follows from \cite{Esnaultpadic}.
When $R$ is equicharacteristic, this follows from \cite{EsnaultXu}.
\end{proof}

\begin{thm}\cite{HogadiXu} \label{thm-HX} \marpar{thm-HX}
Every PAC field of characteristic $0$ is RC solving.
\end{thm}

\begin{thm}\cite[Theorem 1.1]{SPAC} \label{thm-StarrAx}
  \marpar{thm-StarrAx}
A PAC field of characteristic $p$ is RC solving if it contains a
primitive root of unity of order $n$ for every integer $n$ prime to
the characteristic.
\end{thm}

\begin{rmk}
Please note: in \cite{SPAC} the hypothesis that the DVR is
prime regular is missing.  This is a mistake.  The proof is only valid
under the hypothesis that the DVR is prime regular: for the
``bifurcation'' in Step 3 in Section 2 and the use of 
\cite[Lemma 2.5]{GHMS} in the argument preceding Lemma 1.12, the
argument holds 
only if the local homomorphism $S_{\mf{m}_S}\to \OO_{P,Q}$ is
weakly unramified.
For Corollary
1.2, and similar applications, 
this is irrelevant: there is a parameter space over $\SP
\ZZ$ for Fano complete intersections.  
\end{rmk}

\begin{proof}
The proof in \cite{SPAC} depends on a symmetry of
``Bertini theorems''.  Here is a proof that instead uses the finite
field Bertini theorem.  

\mni
By Lemma \ref{lem-ft} and Lemma \ref{lem-extend} (with $d=0$) 
assume we have the
following: a DVR $(\Lambda,\mf{m}_\Lambda)$ whose
residue field is a finite field, a quasi-projective, 
smooth $\Lambda$-scheme $B$ of relative dimension $m$, a projective,
flat morphism $g_B:\Xx_B\to B$, and a field homomorphism
$\text{Frac}(B_0)\hookrightarrow L$ from the fraction field of the
(integral) closed fiber $B_0=B\times_{\SP \Lambda} \SP
(\Lambda/\mf{m}_\Lambda)$ to a perfect PAC field $L$ that is nice.
Replace $\Lambda$ by its Henselization, and base change $B$ and
$\Xx_B$ over the Henselization.  Then there is a unique connected
component of $B$ that contains $B_0$.  Up to replacing $B$ by this
open and closed subset, assume that $B\to \SP \Lambda$ has irreducible
geometric generic fiber.

\mni
Since $B$ is smooth over $\SP \Lambda$, up to replacing $B$ by an open
subset that is dense in all fibers, there exists an \'{e}tale
morphism $f:B\to \AAA^m_\Lambda$.  Denote $X_B\times_B B_0$ by
$X_{B_0}$.  
By the proof of Corollary
\ref{cor-ffBertini}, there exists a finite type, surjective
monomorphism $i:X'_0\to X_0$ such that every connected component of
$X_0'$ is regular and separated.  Up to shrinking $B$ and $B_0$,
assume that the composition $g_0\circ i:X_0'\to B_0$ is flat.  For each
of the finitely many connected components, the algebraic closure of
the field $\Lambda/\mf{m}_\Lambda$ in the function field of the
component is a finite extension obtained by adjoining a root of unity
whose order $n$ is prime to $p$.  Thus the compositum of these fields
is obtained by adjoining a root of unity $\alpha$.  

\mni
Denote by
$m_\alpha(t)$ the minimal polynomial of $\alpha$ over
$\Lambda/\mf{m}_\Lambda$.  This is an irreducible, separable, monic
polynomial.  Let $m(t)\in \Lambda[t]$ be any element reducing to
$m_\alpha(t)$.  Then $\Lambda'=\Lambda[t]/\langle m(t) \rangle$ is an
\'{e}tale extension of $\Lambda$.  Since the quotient
$\Lambda'/\mf{m}_\Lambda \Lambda'$ is a field, the ideal
$\mf{m}_{\Lambda'} = \mf{m}_\Lambda \Lambda'$ is a maximal ideal.  
Since $\mf{m}_\Lambda$ is principal, also $\mf{m}_{\Lambda'}$ is
principal.  Finally, by hypothesis the field $L$ contains all roots of
unity, so $\Lambda/\mf{m}_\Lambda \to L$ factors through
$\Lambda'/\mf{m}_{\Lambda'}$.  Thus, 
up to replacing $\Lambda$ by $\Lambda'$ and replacing
every scheme as above by its base change by the finite and regular local
homomorphism $\Lambda\to \Lambda'$, assume that
$\Lambda/\mf{m}_\Lambda$ is algebraically closed in the fraction field
of every irreducible component of $\Xx_0'$.  Thus the field extension in
Corollary \text{cor-ffBertini} is an isomorphism.  So every field
extension of $\Lambda/\mf{m}_\Lambda$ satisfies the hypothesis of the
corollary.  

\mni
Denote by $k$ the residue field $\Lambda/\mf{m}_\Lambda$.  
If $X'_0\to B_0$ has a PAC section $Y_0\subset X_0'$, then the base
change $Y_0\times_{B_0} \SP L$ is a geometrically irreducible, finite
type $L$-scheme.  Since $L$ is a perfect PAC field, this scheme has an
$L$-point, which then maps under $i$ to an $L$-point of
$X_0\times_{B_0}\SP L$, as desired.  

\mni
Thus, by way of contradiction, assume
that there is no PAC section of $g$.  Then by Corollary \ref{cor-ffBertini}
with $k'$ equal to $k$, there exists a dense, Zariski open subset $U\subset
\text{Grass}_k(\PP^1,\PP^m)$ such that for every field extension
$K/k$ and for every $[L]\in U(K)$, the inverse image $f^{-1}(L)$ is
geometrically irreducible, and the restriction of $g$ over $f^{-1}(L)$
has no PAC section.  By hypothesis, the family $\Xx_B\to B\to
\SP\Lambda$ is a parameter space such that the geometric generic fiber
of $g_B$ is smooth, integral, and separably rationally connected.
Thus, by Theorem \ref{thm-RCfib}, up to replacing $K$ by a finite 
extension $K'/K$ obtained by adjoining a root of unity, there is a
rational section of the restriction of $g$ over $f^{-1}(L)$.  This
contradiction implies that $g$ does have a PAC section.  Therefore
$X_0\times_{B_0} \SP L$ has an $L$-point.
\end{proof}

\section{Specializations of Rationally 
Simply Connected Varieties} \label{sec-RSCspec}  
\marpar{sec-RSCspec}

\mni
The first new result is an analogous result for perfect PAC fields of positive
characteristic that do not necessarily contain all roots of unity, but
where rational connectedness is replaced by \emph{rational simple
  connectedness}.  For an algebraically closed field $\ol{Q}$ 
of characteristic $0$, a \emph{rationally simply connected fibration}
over a $\ol{Q}$-curve is a pair
$(f_{\ol{Q}}:X_{\ol{Q}}\to \Cc_{\ol{Q}},\Ll_{\ol{Q}})$ where
$\Cc_{\ol{Q}}$ is a smooth, projective, connected $\ol{Q}$-curve, where
$f_{\ol{Q}}$ is a proper, flat morphism, and where $\Ll_{\ol{Q}}$ is
an invertible sheaf that satisfies the following six hypotheses.
First, $X_{\ol{Q}}$ is smooth over $\ol{Q}$.  Second, every geometric
fiber of $f_{\ol{Q}}$ is irreducible.  Third, $\Ll_{\ol{Q}}$ is
$f_{\ol{Q}}$-ample.  The final three hypotheses involve the geometric
generic fiber $Y$ of $f_{\ol{Q}}$ is a scheme over the algebraically
closure $k$ of the fraction field of $\Cc_{\ol{Q}}$, together with the
pullback $\Ll_Y$ of $\Ll_{\ol{Q}}$ as an invertible sheaf on $Y$.  The
fourth hypothesis is that for the parameter space $\Kgnb{0,1}(Y/k,1)$
of $1$-pointed, genus $0$ stable maps to $Y$ having $\Ll_Y$-degree
$1$, i.e., ``lines'' $\ell$, for the maximal open subscheme
$Y_{\text{free}}$ 
of $Y$
over which the evaluation morphism
$$
\text{ev}_1:\Kgnb{0,1}(Y/k,1)\to Y, \ \ ([\ell],p)\mapsto p,
$$
is smooth (automatically this is a dense open in $Y$), the 
fiber of $\text{ev}_1$ over every geometric point of $Y_{\text{free}}$ 
is nonempty,
irreducible, and has rationally connected fibers.  There is a second
important morphism, the ``forgetful'' morphism,
$$
\Phi:\Kgnb{0,1}(Y/k,1)\to \Kgnb{0,0}(Y/k,1), \ \ ([\ell],p)\mapsto [\ell].
$$
For every integer
$m\geq 1$ there is a
parameter $k$-scheme $\text{FreeChains}_2(Y/k,m)$ of ordered
$m$-tuples 
$$
(([\ell_1],p_{1,0},p_{1,\infty}),\dots,([\ell]_m,p_{m,0},p_{m,\infty}))
$$
of triples $([\ell_i],p_{i,0},p_{i,\infty})$ of $2$-pointed lines in
$Y$ such that $p_{i,0}$ and $p_{i,\infty}$ are in the dense open
$Y_{\text{free}}$, and such that $p_{i,\infty}$ equals $p_{i+1,0}$ as
points of $Y_{\text{free}}$ for every $i=1,\dots,m-1$.  There is an
evaluation morphism
$$
\text{ev}_2:\text{FreeChains}_2(Y/k,m)\to Y\times_{\SP(k)} Y,
$$
that sends each ordered $m$-tuple as above to the ordered pair
$(p_{1,0},p_{m,\infty})$.  The fifth hypothesis is that there exists
an integer $m\geq 1$ and a dense open $V$ of $Y\times_{\SP(k)}Y$ such
that the fiber of $\text{ev}_2$ over every geometric point of $V$ is
nonempty, irreducible, and ``birationally rationally connected'',
i.e., there exists one projective model of this quasi-projective
variety 
that is rationally connected (hence every projective model is
rationally connected).  If this holds for one $m$, then there exists
an integer $m_0$ such that it holds for every $m\geq m_0$.  The final
hypothesis is that $(Y,\mc{L}_Y)$ contains a very twisting scroll,
cf. \cite[Definition 12.7]{dJHS}.  This hypothesis is equivalent
to existence of a morphism $\zeta:\PP^1_k\to \Kgnb{0,1}(Y/k,1)$ such
that all of the following hold,
\begin{enumerate}
\item[(i)]  the composition $\text{ev}_1\circ \zeta:\PP^1_k\to Y$ is
  free, i.e., $(\text{ev}_1\circ \zeta)^*T_{Y/k}$ is globally
  generated,
\item[(ii)] the morphism $\text{ev}_1$ is smooth at every point in
  $\zeta(\PP^1_k)$,
\item[(iii)] the pullback by $\zeta$ of the relative tangent sheaf
  $T_{\text{ev}_1}$ is ample, and
\item[(iv)] the pullback by $\zeta$ of $T_{\Phi}$ is globally generated.
\end{enumerate}
For a characteristic $0$ field $Q$, a pair $(f_Q:X_Q\to C_Q,\mc{L}_Q)$
is a \emph{rationally simply connected fibration} over a $Q$-curve if
the base change of the pair to the algebraic closure $\ol{Q}$ is a
rationally simply connected fibration over a $\ol{Q}$-curve as above.
The main theorem of \cite{dJHS}, Theorem 13.1 (cf. also
\cite[Definition 4.8 and Theorem 4.9]{SStrsbg}), gives
an integer $\epsilon$ and a sequence
$(Z_{Q,e})_{e\geq \epsilon}$ of irreducible components $Z_{Q,e}$ of
the Hilbert scheme $\text{Hilb}^{et+1-g(C_Q)}_{X_Q/Q}$ satisfying all of the
following.
\begin{enumerate}
\item[(i)]
The geometric generic point of $Z_{Q,e}$ parameterizes the closed
image of a section $\sigma:\Cc_Q\to X_Q$ of $f_K$ of $\Ll_Q$-degree
$e$ and that is $(g)$-free,
i.e., the deformations of the section relative to a fixed divisor of
degree $\max(2g(\Cc_Q),1)$ are unobstructed.
\item[(ii)]
The restriction to $Z_{Q,e}$ of the Abel map of $\Ll_Q$,
$\alpha_{\Ll_Q}|_Z:Z_{Q,e} \to \Pic{e}{\Cc_Q/Q}$, has fiber over the
geometric generic point of $\Pic{e}{\Cc_Q/Q}$ that is nonempty,
irreducible, and rationally connected.
\item[(iii)]
After arbitrary base change from $Q$ to an algebraically closed field,
for every section $\sigma$ as above that is $(g)$-free, after
attaching to the closed image $\sigma(\Cc_Q)$ sufficiently many
general lines in general fibers of $f_Q$, the resulting curve is
parameterized by $Z_{Q,e}$ for the appropriate $\Ll_Q$-degree $e$.
\end{enumerate}

\mni
In order to state the result, it is useful to specify a bounded family
of polarized schemes.  Let $S$ be an integral, Noetherian, regular
scheme of dimension $\leq 1$ whose function field $Q$ has
characteristic $0$; usually $S$ will be a dense Zariski open subset of
the ring of integers $\SP \mf{o}_Q$ of a number field $Q$.  Fix
integers $m,c\geq 1$.

\begin{defn}\cite[Definition 1.9]{StarrXu} \label{defn-pd} \marpar{defn-pd}
A
\emph{parameter datum} over $S$ with a codimension $>c$
compactification 
is a datum 
$$
((M,f_M:\Xx_M\to M,\Ll_{M_Q}),(\ol{M},\OO_{\ol{M}}(1),i))
$$
of a smooth $S$-scheme
$M$ 
of constant relative dimension $m>1$, a
proper, flat morphism $f_M$, an $f_M$-very ample invertible sheaf
$\Ll_{M_Q}$ on $X\times_S \text{Spec}(Q)$, and a codimension $>c$
compactification $(\ol{M},\OO_{\ol{M}}(1),i)$ of $M$ over $S$, i.e.,
$q:\ol{M}\to S$ is flat and proper, $\OO_{\ol{M}}(1)$ is $q$-very
ample, $i:M\to \ol{M}$ is a dense open immersion such that $\partial
\ol{M}:=\ol{M}\setminus M$ intersects every generic fiber of $q$ in a
subscheme of codimension $>c$ in that fiber.
\end{defn}

\mni
For every integer $r\geq 1$, the complete linear system
$H^0(\ol{M}_Q,\OO_{\ol{M}_Q}(r))$ induces a closed immersion into
projective space over $\SP(Q)$.  For the corresponding Grassmannian
parameterizing flat families of linear subvarieties of this projective
space of codimension $m-1$, there is a dense open subset $G_r$
parameterizing linear subvarieties whose inverse images in $\ol{M}_Q$
are smooth, irreducible curves that are complete contained in the open
subset $M_Q$. 

\begin{defn}\cite[Definition 1.10]{StarrXu} \label{defn-RSCprop}
  \marpar{defn-RSCprop} 
A parameter datum as above
satisfies
the \emph{RSC property} if there exists
an 
integer $r_0$ and a sequence $(W_r)_{r\geq r_0}$
of dense open subschemes $W_r \subset G_r$ 
such that for every algebraically closed extension $\ol{K}$ of
$Q$ and for every $\ol{K}$-point of $W_r$ parameterizing a smooth curve
$\Cc_{\ol{K}}$ in $M_{\ol{K}}$, the pullback family $f:\Cc_{\ol{K}}
\times_M\Xx\to \Cc_{\ol{K}}$ 
together with 
the pullback of $\Ll$
is a rationally simply connected fibration over $\Cc_{\ol{K}}$.
\end{defn}  

\mni
The results of \cite{StarrXu} are stated with respect to a
\emph{finite field} $F$.  However, the results hold for any RC
solving field.  When the field is the function field of a curve over
an algebraically closed field, this is already in the proofs of
\cite[Corollary 13.2 and Lemma 13.4]{dJHS}.  

\begin{prop}\cite[Theorem 1.6]{StarrXu} 
\label{prop-concordmain} \marpar{prop-concordmain}
Let $(R,\mf{m}_R)$ be a Henselian DVR that is prime regular 
with characteristic $0$
fraction field $K$, and let $R/\mf{m}_R\to L$
be a field extension.  Let $\Cc_R\to \SP R$ be a generically smooth
$R$-curve.  Let $f_R:\Xx_R\to \Cc_R$ be a projective, surjective morphism,
and let $\Ll_K$ be an $f_R$-ample invertible sheaf on the generic
fiber $X_K =X_R\times_{\SP R} \SP K$.  Let $z:\SP L\to \SP R/\mf{m}_R$
be a field extension that is RC solving.  If $(f_{\ol{K}}:\Xx_{\ol{K}}\to
\Cc_{\ol{K}},\mc{L}_{\ol{K}})$ is a rationally simply connected
fibration, then there exists a sequence $(U_e)_{e\geq \epsilon}$ of
dense open subschemes $U_e\subset \Pic{e}{\Cc_K/K}$ such that
$\Xx_L/\Cc_L$ admits rational points compatibly with $(U_e)_{e\geq
  \epsilon}$.  In particular, for every generic point $\eta$ of the
smooth locus of $\Cc_L=\Cc_R\times_{\SP R} \SP L$, the
$L(\eta)$-scheme $\Xx_R\times_{\Cc_R} \SP L(\eta)$ has an
$L(\eta)$-point.  
\end{prop}

\begin{proof}
The proof in \cite{StarrXu} applies in the RC solving case.
By \cite[Theorem 13.1]{dJHS}, there exists a sequence
$(Z_{R,e})_{e\geq \epsilon}$ of
geometrically integral closed subschemes of the Hilbert scheme
$\text{Hilb}_{X_R/R}$ whose geometric generic point parameterizes the image of a
section of $X_K\to C_K$ and
such that the restriction of the natural Abel map
$\alpha_e:Z_{K,e}\to \Pic{e}{C_K/K}$ is surjective with geometric
generic fiber an integral scheme that is rationally connected.  Since
$K$ has characteristic $0$, by resolution of singularities, 
there exists a projective morphism
$\wt{Z}_{R,e}\to Z_{R,e}$ such that
$\wt{Z}_{K,e}$ is $K$-smooth. Then $U_e$ is defined as the maximal open
subscheme over which $\wt{Z}_{K,e}\to \Pic{e}{C_K/K}$ is smooth and
the fibers intersect the dense open parameterizing closed subschemes
of $X_K$ that are sections of $X_K\to C_K$.  For every $K$-point
$[\mc{A}]$
of
$U_e$ (and these do exist for $e$ sufficiently positive and divisible
using the fact that $R$ is Henselian) the closure
$\wt{Z}_{R,\mc{A}}$
in $\wt{Z}_{R,e}$ of
the fiber of $\alpha_e$ over this $K$-point is a projective, flat
scheme over $\SP R$ whose geometric generic fiber is smooth, integral
and separably rationally connected.  Thus, since $L$ is RC solving,
there exists an $L$-point of $\wt{Z}_{R,\mc{A}}\times_{\SP R} \SP L$.
Using \cite[Lemma 4.1]{StarrXu},  this $L$-point gives an
$L(\eta)$-point of $\Xx_R\times_{\Cc_R} \SP L(\eta)$. 
\end{proof}

\mni
The various complements to the main theorem from \cite{StarrXu}
also extend.  

\begin{prop} \label{prop-concord} \marpar{prop-concord}
Let $S$ be a regular, integral, Noetherian scheme of dimension $\leq
1$ whose function 
field has characteristic $0$ and whose
residue
fields at closed points are either finite fields or characteristic $0$
fields.  Fix a parameter datum over $S$ with a codimension $>1$
compactification.  
Then the following results of
\cite{StarrXu} hold
with the finite field $F$ replaced by any RC solving field $L$: 
Proposition
1.12 and Corollary 1.13.  Assuming that $S$ is a dense open of the
spectrum of the ring of integers of a number field,
Corollary 1.14 also holds.  Finally, for a parameter datum as above,
for a closed point $\mf{s}\in S$ and a regular local homomorphism of
DVRs $\OO_{S,\mf{s}} \to R$, for $\pi_R:\Cc_R\to \SP R$, for an $S$-morphism
$\zeta:\OO_{\Cc_R,\eta} \to \ol{M}$, and for a projective flat
morphism $f_\eta:\Xx_\eta \to \SP \OO_{\Cc_R,\eta}$ as in Proposition
1.17, there exists a modification of the parameter datum over $S$ such
that $\zeta$ extends to a regular morphism $\zeta'$ to $M'$ and such that the
pullback by $\zeta'$ of $\Xx'_{M'}\to M'$ equals $f_\eta$.
In
particular, if the parameter datum satisfies the RSC
property, for every field homomorphism from the residue field of a
closed point of $S$ to $L$, for every function field $E=L(\eta)$ of a
geometrically integral $L$-curve, and for every $S$-morphism $z_E:\SP
E \to M$, the pullback $X_E$ of $X_M$ by $z_E$ has an $E$-rational
point.  
\end{prop}

\begin{proof}
The proof of Proposition 1.12 is essentially Corollary 3.15, which is
valid over any field $F$.  

\mni
The proof of Corollary 1.13 appears to use
the fact that the field is finite, so that the image $z_o$
under
$z:\SP L\Sem{t}\to \ol{M}$ of the closed point is a closed point of
$\ol{M}$.  That can certainly fail if $L$ is not a finite field.
However, since the
construction of the curve $C_e$ is made with respect to the base
change $\ol{M}_L=\ol{M}\times_R\SP L$, the image point $z_0$ in
$\ol{M}_L$ is an
$L$-point, and that is a closed point since $\ol{M}_L$ is a finite
type $L$-scheme.  Let $\Xx_{\ol{M}}\to \ol{M}$ be any projective model
of $\Xx_M\to M\subset \ol{M}$.  
By \cite[Theorem 1.10]{ArtinApprox}, for every integer $e$ there
exists an integral closed curve $C_e$ in 
$\ol{M}_L\times_{\SP L} \PP^1_L$ that approximates to order $e$ the
graph of $z:\SP L\Sem{t} \to \ol{M}_L\times_{\SP L} \PP^1_L$,
considered as a formal section of $\text{pr}_2:\ol{M}_L\times_{\SP L}
\PP^1_L \to \PP^1_L$ over the closed point $0\in\PP^1_L(\SP L)$.
Since $z$ maps the generic point of $L\Sem{t}$ to $M$, also there
exists such $C_e$ that intersects the open $M_L\times_{\SP
  L}\PP^1_L$ in a dense open $C_e^o$; 
in fact this condition will be automatic for all $e\gg
0$.  
Proposition 1.12 then gives a section of the pullback $\Xx_M\times_M
C_e^o \to C_e^o$, which extends to a section of
$\Xx_{\ol{M}}\times_{\ol{M}} C_e \to C_e$ by the valuative criterion
of properness.  Then by \cite[Theorem 1]{Greenberg}, also
$\Xx_{\ol{M}}\times_{\ol{M}} \SP L\Sem{t} \to \SP L\Sem{t}$ has a
section whose generic fiber is a section of $\Xx_M\times_M \SP
L\Semr{t} \to \SP L\Semr{t}$.  

\mni
Corollary 1.14 uses \cite{DenefAxKochen}.  Although the article is
focused on completions at closed points of $\ZZ$,
resp. $\mathbb{F}_p[t]$, the proofs have no hypotheses that the rings
involved should be finite over the residue field, etc.  Using the
beautiful general arguments, the key computation is the isomorphism
$\tau_p:\mc{MR}(\ZZ_p) \to \mc{MR}(\mathbb{F}_p\Sem{t})$.  This is
completely general.  Let $(A,\theta A)$ and $(B, \pi B)$
be Henselian 
DVRs, and let $\ol{\tau}:A/\theta A\to B/\pi B$ be a field
isomorphism.  For every unit $u\in A\setminus \theta A$, there exists
a unit $v\in B \setminus \pi B$ such that $\ol{\tau}(\ol{u})$
equals $\ol{v}$.  Thus, there is a unique bijection
$\tau_{\theta,\pi}:\mc{MR}(A)\to \mc{MR}(B)$ such that for every
integer $n\geq 0$ and every unit $u$,
$\tau_{\theta,\pi}(\text{mres}(u\theta^n))$ equals
$\text{mres}(v\pi^n)$.  This bijection is all that is used in the
Transfer of Residues Lemma, and thus also in the Transfer of
Surjectivity Theorem.  In particular, this applies for $A$ equal to
$L\Sem{t}$ with $\theta$ to $t$ and to $B$ equal to a Cohen ring for $L$
with $\pi$ equal to $p$ (the Witt vectors is one explicit construction
of such a Cohen ring).

\mni
The proof of Proposition 1.17 works in the same way.  In order to
apply the N\'{e}ron desingularization, it is important that the local
homomorphism of DVRs is regular.  
\end{proof}

\mni
To extend \cite[Proposition 1.18]{StarrXu}, we apply Lemma
\ref{lem-extend} to the parameter space of complete intersection
curves in a specified parameter datum.  Thus, let $S$ be the spectrum
of a Dedekind domain such that the fraction field $Q$ of $S$ has
characteristic $0$, and every residue field is a finite field.  For
a specified closed point of $S$, denote by $(\Lambda,\mf{m}_\Lambda)$
the stalk of the structure sheaf at that point.    
Thus $(\Lambda,\mf{m}_\Lambda)$ is a DVR
whose residue field $\kappa$ is a finite field or a characteristic $0$
field.    
Let there be specified a parameter datum over $S$.
Denote by $\Xx_{\ol{M}}\to \ol{M}$ a projective morphism (not
necessarily flat) that restricts to $\Xx_M$ over $M$.

\mni
Assume that the parameter datum satisfies the RSC property.
For
every integer $r\geq r_0$, 
denote by $\Cc_{W_r}\subset
W_r\times_{\SP Q} M_Q$ the universal closed subscheme.  Denote by
$\Cc_{G_r} \subset G_{\Lambda,r}\times_{\SP \Lambda} \ol{M}$ the
closure of $\Cc_{W_r}$.  
Denote the
restrictions to $\Cc_{W_r}$ of the two projections as follows,
$$
\rho_{W_r}:\Cc_{W_r}\to W_r,
$$
$$
\rho_{M,r}:\Cc_{W_r} \to M_Q.
$$
By
hypothesis, $\rho_{W_r}:\Cc_{W_r}\to W_r$ is smooth and
projective with geometric fibers that are irreducible curves.  
For every integer
$e$, denote by $\Pic{e}{\Cc/W_r}\to W_r$ the relative Picard scheme of
$\rho_{W_r}$ 
parameterizing invertible sheaves of degree $d$ on these curves.  The
relative degree over $W_r$ of the pullback to $\Cc_{W_r}$ of
$\OO_{\ol{M}_Q}(1)$ equals $e(r) = e_0 r^{m-1}$ for an integer $e_0$.
For every integer $\ell$, the pullback of $\OO_{\ol{M}_Q}(\ell)$
defines a global section of $\Pic{\ell e(r)}{\Cc/W_r}$ over $W_r$.

\mni
Denote by $f_{r}:\Xx_{r}\to \Cc_{W_r}$ the pullback by $\rho_{M,r}$ of
the universal family $f_{M_Q}:\Xx_{M_Q}\to M_Q$.  
Denote by $\ol{f}_r:\ol{\Xx}_r \to \Cc_{G_r}$ the closure of $\Xx_r$
in $\Cc_{G_r}\times_{\ol{M}} \Xx_{\ol{M}}$.
Denote by $\Ll_r$
the pullback of $\Ll_Q$ to $\Xx_r$.  Consider the base change by the
generic point $\SP(Q(W_r))\to W_r$ of
$\Xx_r$, $\Cc_{W_r}$ and $\Ll_r$.  By the RSC hypothesis, these base
changes form a rationally simply connected fibration.  Thus, by
\cite[Theorem 13.2]{dJHS}, there exists a sequence
$(Z_{e,Q(W_r)})_{e\geq \epsilon'(r)}$ of irreducible components of the
relative Hilbert scheme
$\text{Hilb}^{et+1-g_r}_{\Xx_r/W_r}\times_{W_r} \SP Q(W_r)$
satisfying the conditions listed above.  For every $e\geq
\epsilon'(r)$, define $Z_e$ to be the
closure of $Z_{e,Q(W_r)}$ in the relative Hilbert scheme
$\text{Hilb}^{et+1-g(C_Q)}_{\ol{\Xx}_r/G_r}$.  Define $\epsilon(r)$ to be
the smallest multiple $\ell e(r)$ of $e(r)$ such that $\epsilon(r)\geq
\epsilon'(r)$.

\begin{prop}\cite[Proposition 1.18]{StarrXu}
  \label{prop-concordheight} \marpar{prop-concordheight}
For every
integer $r>0$ there exists an integer $\epsilon(r)>0$ 
with the following
property.  For every closed point of $S$ with residue field $\kappa$,
for every field extension $L/\kappa$ such that $L$ is RC
solving, 
for every integral curve $\Cc_L \subset \ol{M}_L$
intersecting $M_L$, if $\text{deg}_{\Cc_L}(\OO_{\ol{M}}(1)) \leq r$ then there
exists a curve $\mc{C}_L\subset f_{\ol{M}}^{-1}(\Cc_L)$ with $f:\mc{C}_L\to \Cc_L$
an isomorphism over the generic point and with
$\text{deg}_{\mc{C}_L}(\Ll_{\ol{M}}) \leq \epsilon$.
\end{prop}   

\begin{proof}
Having specified $\epsilon(r)$ independent of $\kappa$, we are now
free to prove the proposition after replacing $S$ by $\SP \Lambda$.
Since $Q(W_r)$ has characteristic $0$, by resolution of singularities,
there exists a projective morphism $\wt{Z}_e\to Z_e$ such that the
geometric generic fiber of $\wt{Z}_e\to G_{\Lambda,r}$ is smooth.
There is an Abel map,
$$
\wt{\alpha}_e:\wt{Z}_e\times_{G_r} W_r \to \Pic{e}{\Cc/W_r}.
$$
Denote by $U_e\subset \Pic{e}{\Cc/W_r}$ the maximal open subscheme
over which $\wt{\alpha}_e$ is smooth and such that the fiber of
$\wt{\alpha}_e$ over every geometric point of $U_e$ intersects the
dense open subset of $\wt{Z}_e$ parameterizing closed subschemes that
are images of sections of $f_r$.  In particular, for $e$ equal to
$\epsilon(r)$, 
define $W^o_r\subset W_r$ to be the
image of the dense open $U_e$ under the smooth morphism
$\Pic{e}{\Cc/W_r}\to W_r$.  

\mni
By \cite[Corollary 3.15]{StarrXu}, the curve $\Cc_L$ is an
an irreducible component of multiplicity $1$ in a curve in $\ol{M}_L$
parameterized by a $\Lambda$-morphism $z:\SP L \to G_{\kappa,r}$.  By
Lemma \ref{lem-extend} applied to the parameter space
$G_{\Lambda,r}\to \SP \Lambda$ and the dense open subscheme $W^o_r$ of
the generic fiber $G_{Q,r}$, for $d=0$ and $E=F=L$, there exists an
integral extension of $z$, say
$$
((B\to \SP \Lambda, z_B:B\to G_r),\psi_B:\SP L \to B_0).
$$
Since $\Lambda$ is Henselian, $B$ has geometrically integral closed
fiber and generic fiber over $\SP \Lambda$.  

\mni
Define $R$ to be the
Henselization of the DVR $\OO$ obtained as the stalk of the structure
sheaf of $B$ at the generic point of $B_0$.  Then $\Lambda \to R$ is a
regular local homomorphism, $z_B$ induces a $\Lambda$-morphism,
$z_R:\SP R \to G_R$ mapping the generic point into $W^o_r$, and there
exists a $\Lambda$-morphism $\psi_R:\SP L\to \SP R$ whose composition
with $z_R$ equals $z$.  By construction, for $e=\epsilon(r) = \ell
e(r)$,  
the relative Picard
$\Pic{e}{\CC/W_r}$ has a section over $W_r$ coming from
$\OO_{\ol{M}}(\ell)$.  Thus the pullback by $z_R$ has a section over
$\SP \text{Frac}(R)$.  Since $R$ is Henselian, and since relative
Picard scheme is smooth, there exists a section $[\mc{A}]$ over $R$
that maps into the open subscheme $U_e$.  

\mni
Define $\wt{Z}_{R,\mc{A}}$ 
to be the
closure in $\wt{Z}_e\times_{G_r} \SP R$ of the fiber of
$\wt{\alpha}_e$ over $[\mc{A}]$.  Then $\wt{Z}_{R,\mc{A}}\to \SP R$ is
projective and flat.  By construction, the geometric generic fiber is
separably rationally connected.  Since $R$ is prime regular, and since
$L$ is RC solving, there exists an $L$-point of
$\wt{Z}_{R,\mc{A}}\times_{\SP R} \SP L$.  By \cite[Proposition
4.1]{StarrXu}, the image of this $L$-point in the Hilbert scheme 
$\text{Hilb}^{et+1-g(C_Q)}_{\ol{\Xx}_r/G_r}\times_{G_r} \SP L$
parameterizes a closed subscheme of $\ol{\Xx}_r\times_{G_r}\SP L$
whose restriction over a dense open $\Cc_L^o$ in $\Cc_L$ is the image of a
rational section of $\Xx_M\times_M \Cc_L^o \to \Cc_L^o$.  By
construction, the total degree of this closed subscheme has degree
$e(r)$.  Thus, the closed image of the rational section has degree
$\leq e(r)$.  
\end{proof}

\section{Perfect PAC fields and Global Function Fields} \label{sec-PACglob}  \marpar{sec-PACglob}

\mni
Let $S$ be $\SP \Lambda$, where $\Lambda$ is a Henselian DVR with
finite residue field $\kappa$.  Fix a parameter datum over $S$ with a
codimension $>1$ compactification.  
A \emph{global function field over $M$} is a function field over $M$,
$$
(\SP(R)\to S, E/F,z:\SP E\to M),
$$
such that the residue field extension $\kappa \to F= R/\mf{m}_R$ is
finite, i.e., $F$ is a finite field.  
Thus, the field $E$ -- a function field of a geometrically
integral $F$-curve -- is a global function field.

\begin{defn} \label{defn-global} \marpar{defn-global}
The parameter datum has \emph{rational points over global function
  fields} if for every global function field over $M$ as above, the
base change $X_E = \SP E\times_M X_E$ has an $E$-point.  
\end{defn}

\begin{prop} \label{prop-integral} \marpar{prop-integral}
If a parameter datum has rational points over global function fields,
then for every irreducible closed subset of the closed fiber 
$B\subset M_0$ with pullback family $\Xx_B=B\times_M \Xx_M$, 
there exists a PAC section of $\Xx_B\to B$.  Thus, for every field
extension $\kappa \to L$, for every $S$-morphism $\SP L\to M_0$, the
pullback family $\Xx_L = \SP L\times_M \Xx_M$ has a PAC section over
$\SP L$.  
In particular, if $L$ is a perfect PAC field, then $X_L$
has an $L$-point.
\end{prop}

\begin{proof}
First consider the case that $B$ has dimension $0$, i.e., as a
$\kappa$-scheme this equals $\SP F$ for a finite field extension
$F/\kappa$.  Define $E$ to be $F(t)$, the function field of $\PP^1_F$.
Define $z:\SP E \to M$ to be the composition of $\SP E\to \SP F$ and
the specified closed point $\SP F \to M$.  Since the parameter datum
has rational points over global function fields, there exists a
lifting of $z$ to an $F$-morphism $s:E\to \Xx_B$.  The closure $Y$ of the
image of this morphism is the image of an $F$-morphism from $\PP^1_F$,
and hence $Y$ is geometrically irreducible over $\SP F$.  Thus $Y$ is
a PAC section of $\Xx_B \to B$.

\mni
Thus, without loss of generality, assume that $B$ has dimension $m\geq
1$.  
Define $F$ to be the algebraic closure of $\kappa$ in the
fraction field of $B$.  Thus $B$ is a geometrically integral $F$-scheme.
Let $u:B\to \PP^N_F$ be a generically unramified, finite type
morphism.  Set $c$ equal to $m-1$, and set $r$ equal to $N-c$.  Thus,
$\text{Grass}_F(\PP^r,\PP^N)$ parameterizes linear subspaces of
codimension $c=m-1$, so that the inverse image under $u$ has dimension
$\geq 1$.

\mni
By way of contradiction, assume that $f_B:\Xx_B\to B$ has
no PAC section.  Then by Corollary \ref{cor-ffBertini}, there exists
a finite Galois extension $F'/F$ and a dense open subset $U\subset
\text{Grass}_F(\PP^r,\PP^N)$ such that for every field extension
$K/F$ that is linearly disjoint from $F'/F$, 
for every $[L]\in U(\SP K)$,
the curve $C=B\times_{\PP^N_F} L$ is a geometrically integral curve
over $K$ and for $\Xx_C := \Xx_B \times_B C$, the projection morphism
$f_C:\Xx_C \to C$ has no PAC section.  By elementary considerations of
the intersection of $U$ with any of the open affine spaces in
$\text{Grass}_F(\PP^r,\PP^N)$ forming an open Bruhat cell,
or by the more refined Lang-Weil estimates, there exists an integer
$d_0$ such that 
for every finite field extension $K/F$ of degree $\geq d_0$, $U(K)$ is
nonempty.  In particular, there exists such an extension of the finite
field $F$ of degree that is prime to $d=[F':F]$, e.g., of degree $dd_0+1$.
Since $\Xx_C\to C$ has no PAC section, it also has no rational
section.  This contradicts the hypothesis that the parameter datum has
rational points over all global fields.  This contradiction implies
that $\Xx_B\to B$ does have a PAC section.

\mni
Now for a field extension $\kappa \to L$ and a $S$-morphism $z:\SP L
\to M_0$, define $B\subset M_0$ to be the closure of the image of
$z$.  Then $B$ is an integral closed subscheme of $M_0$.  By the
argument above, $\Xx_B\to B$ has a PAC section.  The base change of
this PAC section by the dominant morphism $\SP L\to B$ is a PAC
section of $\Xx_L\to \SP L$.  
\end{proof}

\mni
By \cite[Proposition 1.12]{StarrXu}, if the parameter
datum satisfies the RSC property, then the parameter datum has
rational points over global function fields.

\begin{cor} \label{cor-integral} \marpar{cor-integral}
For every parameter datum that satisfies the RSC property, the
parameter datum has rational points over global function fields.  Thus,
for every
extension field $\kappa \to L$ for which $L$ is a perfect
PAC field, for every $S$-morphism $z:\SP L \to M$, $X_L$ has an $L$-point.
\end{cor}

%%%%%%%%%%%%%%%%%%%%%%%%%%%%%%%%%%%%%%%%%%%%%%%%%%%%%%%%%%%%%%%
%%
%% Section: The Grothendieck-Serre Conjecture and
%% Serre's "Conjecture II" in Positive Characteristic
%% 
%%%%%%%%%%%%%%%%%%%%%%%%%%%%%%%%%%%%%%%%%%%%%%%%%%%%%%%%%%%%%%%

\section{The Grothendieck-Serre Conjecture and Serre's ``Conjecture
  II'' in Positive Characteristic} \label{sec-GrothSerre}  
\marpar{sec-GrothSerre}

\mni
The proof of 
Serre's ``Conjecture II'' for the function field $k(S)$ of a surface over an
algebraically closed field $k$ of arbitrary characteristic was completed
in \cite[Theorem 1.5]{dJHS}.  Since this is crucial for
establishing Serre's ``Conjecture II'' for function fields of curves
over perfect PAC fields that are nice, we briefly recall the proof and
clarify one point.

\mni
The main step there is the proof of Serre's ``Conjecture II'' for
those semisimple, connected, and simply connected groups over $k(S)$
that are of 
the form $G_0\times_{\SP k} \SP k(S)$, with $G_0$ a semisimple,
connected, and simply connected group over $k$.  Such a group is
necessarily split, since $k$ is algebraically closed.  In particular, since
the automorphism group of the split group $G_0$ of type $E_8$ is
precisely $G_0$ acting by conjugation, it follows that every group
$G$ over $k(S)$ of type $E_8$ is isomorphic to $G_0\times_{\SP k} \SP
k(S)$.  Thus, the split case of Serre's ``Conjecture II'' implies the
$E_8$ case of Serre's ``Conjecture II''.  Combined with tremendous
earlier work on Serre's ``Conjecture II'', the $E_8$ case of Serre's
``Conjecture II'' settles the full case of Serre's ``Conjecture II''
over $k(S)$, when $k$ has \emph{characteristic $0$}, cf. \cite[Theorem
1.2(v)]{CTGP}. 

\mni
In fact, the proof in characteristic $0$ implies the proof in
arbitrary characteristic via the type of lifting results from
\cite{dJS8} and further explored in \cite{StarrXu} and
this article.  

\begin{thm} \label{thm-redcharp} \marpar{thm-redcharp}
The characteristic $0$ case of Serre's ``Conjecture II'' for function
fields of surfaces implies the characteristic $p$ version.  Precisely,
for every algebraically closed field $F$ of characteristic $p$, for
every function field $E/F$ of a geometrically integral $F$-scheme of
dimension $2$, for every semisimple algebraic group $G_E$ over $\SP E$
that is connected and simply connected, for every torsor $\mc{T}_E$
over $\SP E$ for $G_E$, there exists an $E$-point of $\mc{T}_E$.
\end{thm}

\begin{proof}
Let $(\Lambda,\mf{m}_\Lambda)$ be a Henselian DVR whose
residue field $\kappa$ is a finite field of characteristic $p$ 
and whose fraction field $Q$ has
characteristic $0$.  Let $H_0$ be a semisimple
group that is connected and split.  Let $\kappa \to F$ be a field
extension (not yet assumed algebraically closed), 
let $E/F$ be the function field of a geometrically integral
$F$-scheme of dimension $d$ (not yet assumed equal to $2$), 
let $G_E$ be a linear algebraic group
over $E$ such that $G_E\times_{\SP E}\SP \ol{E}$ is isomorphic to
$H_0\times_{\SP \kappa} \SP \ol{E}$ as a group scheme over $\ol{E}$,
and let $\mc{T}_E$ be a $G_E$-torsor over $\SP E$.  
By the proof of \cite[Corollary 1.22]{StarrXu} and Lemma
\ref{lem-extend}, there exists an integral model, 
$$
(B\to \SP \Lambda, \pi^o_B:C^o_B\to B,(G_{\Cc}\to
\Cc^o_B,\mc{T}_{\Cc}\to \Cc^o_B)),
$$
and a pair
$$
(\psi_B:\text{Frac}(B_0)\to F,\psi_E:F(C^o_F)\to E)
$$ 
as in Definition \ref{defn-int}, and where $G_{\Cc}\to \Cc^o_B$, resp.
$\mc{T}_{\Cc}\to \Cc^o_B$,
is a
semisimple group scheme, resp. 
is a torsor for this group scheme, such that the base change by
$\psi_E$ is $(G_E,\mc{T}_E)$.  The choice of parameter datum $M$ for
this integral extension involves the automorphism group scheme of
$H_0$ and is fully explored in \cite[Proof of
Corollary 1.22]{StarrXu}.  

\mni
Now assume that $d$ equals $2$.
The function field $Q(B)$ is
characteristic $0$, so by the characteristic $0$ case of Serre's
``Conjecture II'', there exists a finite extension $Q(B')/Q(B)$ such
that after base change by this extension, the torsor is trivial.
After replacing $B$ by a dense Zariski open whose complement has
codimension $\geq 2$, which thus has nonempty intersection with the
closed fiber $B_0$,
there exists a finite, flat morphism
$B'\to B$ such that $B'$ is integral and the associated extension of
fraction fields is the extension $Q(B')/Q(B)$ above.  It may well
happen that $B'\times_{\SP \Lambda} \SP \kappa$ is reducible.  Choose
one irreducible component $B'_0$, and replace $B'$ by the open
complement of the remaining irreducible components.  Then the residue
field $\kappa(B'_0)$ is a finite algebraic extension of
$\kappa(B_0)$.  

\mni
Assume now that $F$ is algebraically closed.  Then there is a
factorization of $\psi_B:\SP E\to \SP \kappa(B_0)$ through
$\psi_{B'}:\SP E \to \SP \kappa(B'_0)$.  Thus, up to replacing $B$ by
$B'$, assume that the generic fiber of $\mc{T}_{\Cc}\to \Cc^o_B$ is a
trivial torsor over the fraction field $Q(\Cc^o_B)$.  Denote by $\OO$
the DVR that is the stalk of the structure sheaf of $\Cc^o_B$ at the
generic point of the closed fiber $\Cc^o_0=\Cc^o_B\times_{\SP \Lambda}
\SP \kappa$.  Then the pullback of $G_{\Cc}$ and $\mc{T}_{\Cc}$ over
$\SP \OO$ form a semisimple group scheme and a torsor for that group
scheme over a DVR.  By construction, this torsor is trivial when
restricted to the fraction field $Q(\Cc^o_B)$ of $\OO$.  Thus, by
Nisnevich's solution of the Grothendieck-Serre Conjecture over DVRs,
\cite{Nisnevich}, the restriction of the torsor over the
residue field of $\OO$ is also trivial.  Taking the further base
change by $\psi_E$, it follows that the original torsor $\mc{T}_E$
over $\SP E$ is trivial.
\end{proof}

%%%%%%%%%%%%%%%%%%%%%%%%%%%%%%%%%%%%%%%%%%%%%%%%%%%%%%%%%%%%%%%
%%
%% Section: Proofs of Theorems \ref{thm-SerreIIpac},
%% \ref{thm-PeriodIndexpac}, \ref{thm-C2nicepac}
%% 
%%%%%%%%%%%%%%%%%%%%%%%%%%%%%%%%%%%%%%%%%%%%%%%%%%%%%%%%%%%%%%%

\section{Proofs of Theorems \ref{thm-SerreIIpac}, 
\ref{thm-PeriodIndexpac},\ref{thm-C2nicepac}}
\label{sec-proofs} \marpar{sec-proofs}

\mni
By Theorem \ref{thm-HX} and Theorem \ref{thm-StarrAx}, every perfect
PAC field $L$ that is nice is RC solving.  Thus, by Proposition
\ref{prop-concord}, for every regular, integral 
Noetherian scheme $S$ of dimension $\leq
1$ whose function field has characteristic $0$, for every parameter
datum over $S$ with a codimension $>1$ compactification that satisfies
the RSC property, for every function field $E=L(\eta)$ of a
geometrically integral $L$-curve, for every $S$-morphism $z_E:\SP E\to
M$, the pullback $\Xx_E$ of the universal family $\Xx_M\to M$ by $z_E$
has an $E$-rational point.  

\mni
\textbf{The $C_2$ Property and the Period-Index Theorem.}
By \cite[Proposition 1.19]{SPAC}, there is
a parameter datum as above whose universal family is the family of
$2$-Fano complete intersections in projective space, resp. the family
of minimal homogeneous varieties.  Thus, every function field
$L(\eta)$ of a geometrically integral curve over a perfect PAC field
that is nice is $C_2$.  One example of minimal homogeneous spaces are
generalized Severi-Brauer varieties, i.e., a smooth projective $E$-scheme
$X_E$ such that $X_E\otimes_{\SP E}\SP \ol{E}$ is isomorphic to a
(standard $A_n$-type) Grassmannian $\text{Grass}_{\ol{E}}(r,\ol{E}^n)$ 
whose associated Isom torsor
reduces to the group of inner automorphisms (there are outer
automorphisms only if $n$ equals $2r$, $r>1$). 
Thus, each generalized Severi-Brauer
variety $\Xx_E$ has an $E$-point if (and only if) it has 
vanishing elementary obstruction, i.e.,
if there exists an invertible sheaf $\Ll_E$ on $\Xx_E$ whose base change
to $\Xx_E\times_{\SP E} \SP \ol{E}$ generates the Picard group.  As
explained in \cite[Theorem 11.1]{SStrsbg} (based on
joint work with de Jong from \cite{dJS8}), this implies that
Period equals Index for Severi-Brauer varieties over $E$.

\mni
\textbf{The Split Case of Serre's ``Conjecture II''.  Full Serre's
  ``Conjecture II'' in Characteristic $0$.}
Next, as explained in the proof of \cite[Theorem 1.4]{dJHS},
existence of $E$-rational points on minimal homogeneous varieties
implies Serre's ``Conjecture II'' for torsors over $E$ for semisimple
groups that are connected, simply connected, and \emph{split}.  As
explained in the proof of Theorem \ref{thm-redcharp}, when $E$ is of
characteristic $0$, this implies the full Serre's ``Conjecture II''.
Thus, it only remains to prove Serre's ``Conjecture II'' in case $L$
has characteristic $p$.  Since $L$ is nice, $L$ contains the algebraic
closure $\ol{\kappa}$ of the prime field.

\mni
\textbf{Full Serre's ``Conjecture II'' in Positive Characteristic.}
By the proof of Theorem \ref{thm-redcharp}, for every semisimple
algebraic group $G_E$ over $E$ that is connected and simply connected,
and for every $G_E$-torsor $\mc{T}_E$, there exists an integral
model.  In particular, the closed fiber of this integral model gives a
smooth, quasi-projective $\kappa$-scheme $B_0$ that is integral (but
typically not 
geometrically irreducible), a quasi-projective, smooth morphism $C_0^o\to B_0$
whose geometric fibers are irreducible curves, a semisimple group
scheme $G_{\Cc,0}\to \Cc_0^o$ whose geometric fibers are connected and
simply connected, and a torsor $\mc{T}_{\Cc,0}\to \Cc_0^o$ under
$G_{\Cc,0}$.  Moreover, there is a homomorphism of $\kappa$-extensions 
$\psi_B:\kappa(B_0)
\to L$ and an associated isomorphism of $F$-extensions
$\psi_E:F(\Cc_F^o)\to E$ such that the base changes by $\psi_E$ of
$G_{\Cc,0}$, resp. $\mc{T}_{\Cc,0}$ equal $G_E$, resp. $\mc{T}_E$.

\mni
Since $B_0$ is a quasi-projective, smooth, irreducible
$\kappa$-scheme, $\kappa(B)/\kappa$ is a finitely generated field
extension.  Thus the algebraic closure $\kappa'$ 
of $\kappa$ in $\kappa(B)$ is a
finite extension of $\kappa$, and so it is again a finite field.  Note
that $B$ is geometrically integral over $\kappa'$.  Via the morphisms
to $B_0$, the schemes $\Cc_0^o$, $G_{\Cc,0}$ and $\mc{T}_{\Cc,0}$ are
all quasi-projective $\kappa'$-schemes.

\mni
Denote by $\Cc_0$ a projective compactification of $\Cc_0^o$.  Up to
replacing $B_0$ by a dense open subscheme, assume that $\Cc_0\to B_0$ is
projective and flat.  Also, up to normalizing (all schemes are of
finite type over $\kappa'$, so normalization is finite), 
assume that $\Cc_0^o$ is normal.  Then the
non-regular locus has codimension $2$.  Since $\Cc_0$ has relative
dimension $1$ over $B_0$, the image of the non-regular locus in $B_0$
is a closed subset of codimension $\geq 1$.  Up to shrinking $B_0$
once more, assume that $\Cc_0$ is regular.  

\mni
Denote by $\ol{\mc{T}}_{\Cc,0}$ a projective compactification of
$\mc{T}_{\Cc,0}$.  As above, the non-flat locus of
$\ol{\mc{T}}_{\Cc,0}\to \Cc_0$ has codimension $\geq 2$ in $\Cc_0$.
The image of this subset in $B_0$ is a closed subset of codimension
$\geq 1$.  Thus, up to shrinking $B_0$ once more, assume that
$\ol{\mc{T}}_{\Cc,0} \to \Cc_0$ is projective and flat.

\mni
Inside the relative Hilbert scheme
$\text{Hilb}_{\ol{\mc{T}}_{\Cc}/B_0}$, there is a locally closed
subscheme $\text{Sec}$ such that for every $B$-scheme $T$, the
$T$-points of $\text{Sec}$ are precisely those closed subschemes $\ol{Z}\subset
T\times_{B_0} \ol{\mc{T}}_{\Cc,0}$ satisfying the following conditions:
for the intersection
$Z$ of $\ol{Z}$ with the open subscheme $T\times_{B_0} \mc{T}_{\Cc,0}$,
for the projection $Z\to T\times_{B_0} \Cc_0$, the maximal open subset
of $T\times_{B_0}\Cc_0$ over which this projection is an isomorphism
is an open subset that surjects to $T$.  Denote by $Z^o$ the inverse
image in $Z$ of this open subset, so that $Z^o\to T\times_{B_0} \Cc_0$
is an open immersion.  Thus, $Z^o$ defines a rational section of the
base change morphism,
$$
\mc{T}_{\Cc,0}\times_{B_0}\text{Sec} \to \Cc_0\times_{B_0} \text{Sec}.
$$
Of course the scheme $\text{Sec}$ has countably
many connected components $(\text{Sec}_i)_{i\in I}$  
indexed by the countably many possible
Hilbert polynomials of $Z$. 

\mni
The claim is that there exists $i\in I$ such that the morphism
$$
\phi_i:\text{Sec}_i\times_{\SP \kappa'} \SP \ol{\kappa}' \to B_0\times_{\SP
  \kappa'} \SP \ol{\kappa}',
$$
has a PAC section.  
Assuming the claim, since $L$ is assumed to
contain $\ol{\kappa}'$, the base change $\SP L \times_{B_0} \text{Sec}$
has an $L$-point.  For this $L$-point, the rational section $Z^o$
defines an $E$-point of $\mc{T}_E$.  Thus, it suffices to prove the
claim.  

\mni
If $B_0$ has dimension $0$, then the claim is true.  Indeed, $B_0\to
\SP \kappa'$ is an isomorphism, and
$\Cc_0\to B_0$ is a geometrically integral curve over the finite field
$\kappa'$.
By Steinberg's Theorem, \cite{Steinberg}, or 
by \cite{dJS},
after base change to $\SP \ol{\kappa}',$
the pullback of the torsor $\mc{T}_{\Cc,0}$ over this curve has
a rational section.  Thus, assume that $B$ has dimension $m\geq 1$.

\mni
The existence or non-existence of a PAC section of $\phi_i$ 
is preserved by base change from the algebraically closed field
$\ol{\kappa}'$ by any extension $\ol{\kappa}'\hookrightarrow k$ with $k$
an algebraically closed field.  
Let $u:B_0\to \PP^N_{\kappa'}$ be a generically unramified, finite type
morphism.  Define $c$ to be $m-1$, and define $r$ to be $N-c$.
Define $k$ to be the algebraic closure of the function field of
$\text{Grass}_{\kappa'}(\PP^r,\PP^N)$.  Define $[H]\in
\text{Grass}_{\kappa'}(\PP^r,\PP^N)(\SP k)$ to be the $k$-point
corresponding to the universal linear space.  By the proof of Theorem
\ref{thm-ffBertini}, the corresponding curve
$B_H=B\times_{\PP^N_{\kappa'}} H$ is a smooth, irreducible, 
quasi-projective curve
over $k$.

\mni
The base change of $\Cc_0\to B_0$ by $B_H\to B_0$ gives a projective,
flat, generically smooth morphism $\Cc_H\to B_H$ with geometrically
irreducible fibers.  Since $B_H$ is
itself a smooth, irreducible, 
quasi-projective curve over the algebraically closed
field $k$, $\Cc_H$ is a generically smooth, irreducible, 
quasi-projective surface
over $k$.  The pullback of
$\mc{T}_{\Cc,0}$ over the surface $\Cc_H$ is a torsor for a
semisimple, connected, and simply connected algebraic group over
$\Cc_H$.  By Theorem \ref{thm-redcharp}, there exists a rational
section of this torsor.  The closure of this section in the pullback of
$\ol{\mc{T}}_{\Cc,0}$ defines a closed subscheme $Z$.  The projection
$Z\to B_H$ is flat, since $B_H$ is a smooth curve.  Over a dense open
subset of $B_H$, this closed subscheme defines a section of the
restriction of $\phi_i$, for $i$ equal to the Hilbert polynomial of
the fibers of $Z\to B_H$.

\mni
Note that $\ol{\kappa}'$ is already algebraically closed, so every
finite Galois extension is an isomorphism, thus automatically linearly
disjoint from every field extension $K/k$.
For this choice of $i$, if there is no PAC section of $\phi_i$, then by
Corollary \ref{cor-ffBertini}, there exists a dense, Zariski open
subset $U_i \subset \text{Grass}_{\ol{\kappa}}(\PP^r,\PP^N)$ such that
for every field 
extension $K/\kappa$ 
and for every $[H']\in U_i(\SP K)$, there is no PAC section of
the restriction of $\phi_i$ over the curve $B_0\times_{\PP^N_{\kappa'}} H'.$
For $K$ equal to $k$ and for $[H]$, there is a PAC section of
the restriction of $\phi_i$ over $B_0\times_{\PP^N_\kappa} H$.
Therefore, there does exist a PAC section of $\phi_i$.  Thus, the
torsor $\mc{T}_E$ has an $E$-point.

\mni
\textbf{Low Degree Complete Intersections in Grassmannians.}
By \cite[Proposition 1.19]{StarrXu}, there exists a parameter
datum $M$ with a codimension $>1$ compactification for pairs $(\Xx,\Ll)$
of polarized schemes whose base change to an algebraic closure is
isomorphic to $\text{Grass}(r,K^{\oplus n})$ with its Pl\"{u}cker
invertible sheaf.  Denote by $P\to M$ the projective bundle
parameterizing $1$-dimensional subspaces of $H^0(\Xx,\Ll^{\otimes
  d})$.  This admits a codimension $>1$ compactification
$P\hookrightarrow \ol{P}$ by the same
GIT construction used to construct the codimension $>1$
compactification of $M$.
Inside $\Xx_M\times_M P$, define $Y_P$ to be the closed
subscheme that is the zero scheme of the corresponding global section
of $\Ll^{\otimes d}$.  The projection $Y_P\to \Xx_M$ is a projective
bundle.  Thus $Y_P$ is smooth over $S$.  Thus, by \cite[Lemma
4.4]{StarrXu}, the datum satisfies the first hypothesis of the 
RSC property.  The inequality $(3r-1)d^2-d<n-4r-1$ implies the
inequality $r(n-r) \geq 3$.  Thus, by \cite[Corollary 4.6]{StarrXu},
the datum satisfies the second hypothesis of the RSC property. 
The third hypothesis of the RSC property follows by construction:
since $\Ll$ is relatively ample for $\Xx_M\to M$, it is also
relatively ample for $\Xx_M\times_M P \to P$, and thus its restriction
to $Y_P$ is relatively ample for $Y_P\to P$.  Finally, hypotheses
four, five, and six are verified in the PhD thesis of Robert Findley,
\cite{Findley}.

%%%%%%%%%%%%%%%%%%%%%%%%%%%%%%%%%%%%%%%%%%%%%%%%%%%%%%%%%%%%%%%
%%
%% Section: Proof of Theorems \ref{thm-AKnicepac}
%% 
%%%%%%%%%%%%%%%%%%%%%%%%%%%%%%%%%%%%%%%%%%%%%%%%%%%%%%%%%%%%%%%

\section{Proof of Theorem \ref{thm-AKnicepac}}
\label{sec-proofsAK} \marpar{sec-proofsAK}

\mni
Let $(R,\mf{m}_R)$ be a Henselian DVR with residue field $k$ (no
hypothesis yet on $k$) and with fraction field $K$.  Assume that $R$
is excellent, i.e.,
$\text{Frac}(\wh{R})/K$ is a separable extension.  This holds
automatically if $K$ has characteristic $0$ or if $R$ is complete.

\begin{lem} \label{lem-CDtrans} \marpar{lem-CDtrans}
If the residue field $k$ is separably closed, then $K$ has cohomological 
dimension $\leq
1$.  If $k$ is algebraically closed, then $K$ has
dimension $\leq 1$, and it 
is even $C_1$.
If the cohomological dimension of $k$ is $\leq 1$, then the
cohomological dimension of $K$ is $\leq 2$ under either of the
following conditions: if $K$ has characteristic $0$ or if $R$ is complete.  
If $k$ is
perfect of
dimension $\leq 1$, then for every finite extension $K'/K$,
for every Severi-Brauer variety
over $K'$, the Period equals the Index.    
\end{lem}

\begin{proof}
Let $K'/K$ be any finite algebraic extension.  Since $R$ is excellent,
the integral closure $R'$ of $R$ in $K'$ is a finite $R$-module.
Moreover, $R'$ is a semilocal ring whose localization at any maximal
ideal is a DVR.  The residue fields of $R'$ are finite extensions of
$k$.  The above hypotheses on $k$ are each preserved under finite
field extension.  Thus, the localizations of $R'$ satisfy the same
(respective) hypotheses as $R$.  Thus, every argument below for $K$
also applies to $K'$.

\mni
If $K$ is complete, and if $k$ has cohomological dimension $\leq 0$,
resp. $\leq 1$, then also $K$ has cohomological dimension $\leq
1$. resp. $\leq 2$, \cite[Proposition II.12, p. 85]{GalCoh}.  

\mni
Next
consider the case that $K$ is not complete.  For every prime $\ell$
different from the characteristic of $K$, $\text{cd}_\ell(K)\leq 1$ if
and only if $\text{Br}(K)[\ell]=\{0\}$, \cite[Proposition II.4,
p. 76]{GalCoh}.  Given a Severi-Brauer variety $X_K\to \SP K$ whose
period equals $\ell$, there is a projective, flat model $X_R\to \SP R$
(typically ramified over the closed point).  If $\text{Frac}(\wh{R})$
has cohomological dimension $\leq 1$, then $X_R\times_{\SP R} \SP
\wh{R}$ has an $\wh{R}$-point.  Then by approximation \cite[Theorem
1]{Greenberg}, also $X_R$ has an $R$-point.  If $K$ has characteristic
$p$, then the $p$-cohomological dimension of $K$ is $\leq 1$,
\cite[Proposition II.3, p. 75]{GalCoh}.  Thus, if
$\text{Frac}(\wh{R})$ has cohomological dimension $\leq 1$, then also
$K$ has cohomological dimension $\leq 1$.  Please note: this argument
does not say anything about $\text{Br}(K)[p]$.

\mni
By \cite{Lang52}, if $R$ is complete and $k$ is algebraically closed,
then $K$ is a $C_1$-field.  Again applying \cite[Theorem 1]{Greenberg}, this
also holds if $R$ is excellent but not necessarily complete.  

\mni
Next assume that $K$ has characteristic $0$
and that $k$ has cohomological dimension $\leq 1$.  Then
$\text{Frac}(\wh{R})$ has cohomological dimension $\leq 2$. 
By the
Merkurjev-Suslin Theorem, \cite[Corollary 24.9]{Suslin84}, the field
$K$, resp. $\text{Frac}(\wh{R})$, has cohomological dimension $\leq 2$
if and only if, for every central simple algebra over $K$,
resp. over $\text{Frac}(\wh{R})$, the reduced norm is surjective.  For
a central simple algebra $A_K$ over $K$, this extends to a finite,
flat $R$-algebra $A_R$ (possibly ramified at the closed point).  For
every $r\in R$, the equation $\text{Nrd}_{A_R/R}(x) = r$ have a
solution in $\wh{R}$, since $\text{Frac}(\wh{R})$ has cohomological
dimension $\leq 2$.  Thus, again applying \cite[Theorem 1]{Greenberg},
also there is a solution over $R$.  Thus $K$ has cohomological
dimension $\leq 2$.

\mni
Finally, assume that $k$ is perfect.  Since $R$ is Henselian, for
every $n\geq 0$, the pullback map 
$$
H^n_{\text{\'{e}t}}(\SP k,\mathbb{G}_m) 
\to
H^n_{\text{\'{e}t}}(\SP R,\mathbb{G}_m) 
$$
is an isomorphism.  
By 
\cite[Proposition 2.1, p. 93]{BrauerIII},
there is a restriction exact
sequence 
$$
\begin{CD}
0 @>>> \text{Br}(k) @>>> \text{Br}(K) @>>> 
H^1_{\text{\'{e}t}}(\SP
k,\QQ/\ZZ) 
@>>> H^3_{\text{\'{e}t}}(\SP k,\mathbb{G}_m) \dots
\end{CD}
$$
If $k$ has cohomological dimension $\leq 1$, this gives an
isomorphism,
$$
\textbf{Br}(K)\xrightarrow{\cong}
H^1_{\text{\'{e}t}}(\SP
k,\QQ/\ZZ).  
$$
Via the short exact sequence of Abelian groups,
$$
\begin{CD}
0 @>>> (1/d)\ZZ/\ZZ @>>> \QQ/\ZZ @> d >> \QQ/\ZZ @>>> 0,
\end{CD}
$$
the $d$-torsion in the Brauer group is identified with 
$H^1_{\text{\'{e}t}}(\SP
k,(1/d)\ZZ/\ZZ).$  Thus, for every element $\alpha$ of order $d$ in
the Brauer group, the associated torsor over $k$ gives a Galois field
extension $k'/k$ with cyclic Galois group of order $d$.  There is an
associated \'{e}tale extension $R\to R'$ of degree $d$ that is also a
local homomorphism of DVRs.  The pullback of $\alpha$ to $\SP k'$ is
the zero class.  Thus, comparing exact sequences for $R$ and $R'$, the
pullback of $\alpha$ to $R'$ is the zero class.  Thus, the index of
$\alpha$ equals $d$.
\end{proof}

\mni
Let $S$ be a dense open subscheme of the spectrum of the ring of
integers of a number field.  For every parameter datum over $S$, 
by the extension Proposition
\ref{prop-concord} of \cite[Corollary 1.4]{StarrXu}, which in turn
relies on \cite[Section 7]{DenefAxKochen}, 
there exists an integer $p_0$ such
that for 
every closed point $\SP \kappa \to S$ of characteristic $p\geq p_0$,
for every field extension $\kappa\hookrightarrow L$ with $L$ a perfect
PAC field that is nice, for every pair of Henselian DVR over $S$,
$\SP A\to S$ with $\mf{m}_A = \theta A$, resp.  
$\SP R\to S$ with $\mf{m}_R = \pi R$, each having isomorphic 
residue field extensions of $\kappa$,
$\ol{\tau}:A/\mf{m}_A\xrightarrow{\cong} R/\mf{m}_R$, both equal to
$\kappa \hookrightarrow L$,
for every $S$-morphism $z_A: \SP A\to M$, the pullback
$z_A^*\Xx_M$ has an $A$-point if and only if for every
$S$-morphism $z_R:\SP R\to M$, the pullback $z_R^*\Xx_M$ has an
$R$-point.  This uses the following isomorphism of commutative monoids
under multiplication
$\mc{MR}(A) = A/(1+\mf{m}_A)$, resp. $\mc{MR}(B)=B/(1+\mf{m}_B)$, 
given by
$$
\tau_{\theta,\pi}:\mc{MR}(A)\to \mc{MR}(R), \ \
\tau_{\theta,\pi}(\text{mres}(u\theta^n)) = \text{mres}(v\pi^n),
$$ 
where $u\in A\setminus \mf{A}$, resp. $v\in R\setminus\mf{R}$, are units
such that $\ol{\tau}(\ol{u})$ equals $\ol{v}$ as elements in $L$.  

\mni
Assume now that the parameter datum has a codimension $>1$
compactification and has the RSC property.
By the extension of \cite[Corollary 1.13]{StarrXu}, for
$A=L\Sem{t}$, for every morphism $z_A:\SP L\Sem{t}\to M$, $z_A^*\Xx_M$
does have an $A$-point.  Thus, for every Henselian DVR $R$ over $S$
with residue field extension $\kappa\hookrightarrow L$, for every
$S$-morphism $z_R:\SP R\to M$, also $z_R^*\Xx_M$ has an $R$-point.  

\mni
Applying this to the parameter datum from the previous section 
for complete intersections in
projective space, resp. hypersurfaces in Grassmannians (both of which
are defined over $\SP \ZZ$), 
for every $(n;d_1,\dots,d_c)$ with $d_1^2+\dots + d_c^2 < n-1$, 
resp. for every $(n,r,d)$ with $(3r-1)d^2 - d < n-4r-1$, there exists
an integer $p_0$ such that for every $p\geq p_0$,
for every Henselian DVR $R$ with residue field $L$ a perfect PAC field
of characteristic $p$ that contains a primitive root of unity of order
$n$ for every integer $n$ prime to $p$, there is a $K$-point of every 
$(\Xx_K,\Ll_K)$ over
$\SP K$ whose base change to $\ol{K}$ is the common zero scheme in
$\PP^{n-1}_{\ol{K}}$ of hypersurfaces of degrees $(d_1,\dots,d_c)$
together with the restriction of $\OO_{\PP^{n-1}}(1)$, resp. the base
change is isomorphic to a degree $d$ hypersurface in
$\text{Grass}_{\ol{K}}(r,\ol{K}^{\oplus n})$ together with its
Pl\"{u}cker invertible sheaf.  

\mni
Finally, applying this to 
the parameter datum for torsors for the split group of
type $E_8$, there exists an integer $p_0$ such that for every $p\geq
p_0$, for every DVR as above with $L$ of characteristic $p\geq p_0$,
every torsor over $\SP K$ for the split group of type $E_8$ is
trivial.  Since the center of this group is trivial, and since there
are no outer automorphisms for this group, also every torsor for the
automorphism group of this group is a trivial torsor.  Thus every form
of $E_8$ over $\SP K$ is isomorphic to the split form.  Therefore,
every torsor for a group of type $E_8$ over $\SP K$ is a trivial
torsor.  By Lemma \ref{lem-CDtrans} 
also the characteristic $0$ field $K$ has cohomological dimension $2$ and
satisfies Period equals Index.  By \cite[Theorem 1.2]{CTGP}, Serre's
``Conjecture II'' holds for $K$.

%%%%%%%%%%%%%%%%%%%%%%%%%%%%%%%%%%%%%%%%%%%%%%%%%%%%%%%%%%%%%%%
%%
%% Section: Proof of Theorems \ref{thm-C2pac}
%% 
%%%%%%%%%%%%%%%%%%%%%%%%%%%%%%%%%%%%%%%%%%%%%%%%%%%%%%%%%%%%%%%

\section{Proof of Theorem \ref{thm-C2nicepac}}
\label{sec-proofs2} \marpar{sec-proofs2}

\mni
For $2$-Fano complete intersections in projective space, resp. for low
degree hypersurfaces in Grassmannians, there exists a parameter datum
for these schemes
over $\SP \ZZ$ with a codimension $>1$ compactification, and the
parameter datum satisfies the RSC property, cf. the previous section.
Thus, by Corollary \ref{cor-integral}, every perfect PAC field of
positive characteristic has rational points for these schemes.

\bibliography{my}

\begin{thebibliography}{GHMS05}

\bibitem[Art69]{ArtinApprox}
M.~Artin.
\newblock Algebraic approximation of structures over complete local rings.
\newblock {\em Inst. Hautes \'Etudes Sci. Publ. Math.}, (36):23--58, 1969.

\bibitem[Ax68]{Ax}
James Ax.
\newblock The elementary theory of finite fields.
\newblock {\em Ann. of Math. (2)}, 88:239--271, 1968.

\bibitem[CTGP04]{CTGP}
J.-L. Colliot-Th{\'e}l{\`e}ne, P.~Gille, and R.~Parimala.
\newblock Arithmetic of linear algebraic groups over 2-dimensional geometric
  fields.
\newblock {\em Duke Math. J.}, 121(2):285--341, 2004.

\bibitem[Den16]{DenefAxKochen}
Jan Denef.
\newblock Geometric proofs of theorems of {A}x-{K}ochen and {E}r\v sov.
\newblock {\em Amer. J. Math.}, 138(1):181--199, 2016.

\bibitem[dJHS11]{dJHS}
A.~J. de~Jong, Xuhua He, and Jason~Michael Starr.
\newblock Families of rationally simply connected varieties over surfaces and
  torsors for semisimple groups.
\newblock {\em Publ. Math. Inst. Hautes \'Etudes Sci.}, (114):1--85, 2011.

\bibitem[dJS03]{dJS}
A.~J. de~Jong and J.~Starr.
\newblock Every rationally connected variety over the function field of a curve
  has a rational point.
\newblock {\em Amer. J. Math.}, 125(3):567--580, 2003.

\bibitem[dJS07]{dJS9}
A.~J. de~Jong and Jason Starr.
\newblock Higher {F}ano manifolds and rational surfaces.
\newblock {\em Duke Math. J.}, 139(1):173--183, 2007.

\bibitem[Efr01]{Efrat}
Ido Efrat.
\newblock A {H}asse principle for function fields over {PAC} fields.
\newblock {\em Israel J. Math.}, 122:43--60, 2001.

\bibitem[Esn07]{Esnaultpadic}
H{\'e}l{\`e}ne Esnault.
\newblock Coniveau over {$\mathfrak p$}-adic fields and points over finite
  fields.
\newblock {\em C. R. Math. Acad. Sci. Paris}, 345(2):73--76, 2007.

\bibitem[EX09]{EsnaultXu}
H{\'e}l{\`e}ne Esnault and Chenyang Xu.
\newblock Congruence for rational points over finite fields and coniveau over
  local fields.
\newblock {\em Trans. Amer. Math. Soc.}, 361(5):2679--2688, 2009.

\bibitem[Fin10]{Findley}
Robert~Adam Findley.
\newblock {\em Rational curves in low degree hypersurfaces of {G}rassmannian
  varieties}.
\newblock ProQuest LLC, Ann Arbor, MI, 2010.
\newblock Thesis (Ph.D.)--State University of New York at Stony Brook.

\bibitem[FJ05]{FriedJarden}
Michael~D. Fried and Moshe Jarden.
\newblock {\em Field arithmetic}, volume~11 of {\em Ergebnisse der Mathematik
  und ihrer Grenzgebiete. 3. Folge. A Series of Modern Surveys in Mathematics
  [Results in Mathematics and Related Areas. 3rd Series. A Series of Modern
  Surveys in Mathematics]}.
\newblock Springer-Verlag, Berlin, second edition, 2005.

\bibitem[GHMS05]{GHMS}
Tom Graber, Joe Harris, Barry Mazur, and Jason Starr.
\newblock Rational connectivity and sections of families over curves.
\newblock {\em Ann. Sci. \'Ecole Norm. Sup. (4)}, 38:671--692, 2005.

\bibitem[GHS03]{GHS}
Tom Graber, Joe Harris, and Jason Starr.
\newblock Families of rationally connected varieties.
\newblock {\em J. Amer. Math. Soc.}, 16(1):57--67 (electronic), 2003.

\bibitem[Gre66]{Greenberg}
Marvin~J. Greenberg.
\newblock Rational points in {H}enselian discrete valuation rings.
\newblock {\em Inst. Hautes \'Etudes Sci. Publ. Math.}, (31):59--64, 1966.

\bibitem[Gro68]{BrauerIII}
Alexander Grothendieck.
\newblock Le groupe de {B}rauer. {III}. {E}xemples et compl\'ements.
\newblock In {\em Dix {E}xpos\'es sur la {C}ohomologie des {S}ch\'emas}, pages
  88--188. North-Holland, Amsterdam, 1968.

\bibitem[GS06]{GilleSzamuely}
Philippe Gille and Tam\'as Szamuely.
\newblock {\em Central simple algebras and {G}alois cohomology}, volume 101 of
  {\em Cambridge Studies in Advanced Mathematics}.
\newblock Cambridge University Press, Cambridge, 2006.

\bibitem[GS13]{GS}
Tom Graber and Jason~Michael Starr.
\newblock Restriction of sections for families of abelian varieties.
\newblock In {\em A celebration of algebraic geometry}, volume~18 of {\em Clay
  Math. Proc.}, pages 311--327. Amer. Math. Soc., Providence, RI, 2013.

\bibitem[HX09]{HogadiXu}
Amit Hogadi and Chenyang Xu.
\newblock Degenerations of rationally connected varieties.
\newblock {\em Trans. Amer. Math. Soc.}, 361(7):3931--3949, 2009.

\bibitem[Jou83]{Jou}
Jean-Pierre Jouanolou.
\newblock {\em Th\'eor\`emes de {B}ertini et applications}, volume~42 of {\em
  Progress in Mathematics}.
\newblock Birkh\"auser Boston Inc., Boston, MA, 1983.

\bibitem[JP09]{JardenPop}
Moshe Jarden and Florian Pop.
\newblock Function fields of one variable over {PAC} fields.
\newblock {\em Doc. Math.}, 14:517--523, 2009.

\bibitem[KMM92]{KMM}
J{\'a}nos Koll{\'a}r, Yoichi Miyaoka, and Shigefumi Mori.
\newblock Rationally connected varieties.
\newblock {\em J. Algebraic Geom.}, 1(3):429--448, 1992.

\bibitem[Lan52]{Lang52}
Serge Lang.
\newblock On quasi algebraic closure.
\newblock {\em Ann. of Math. (2)}, 55:373--390, 1952.

\bibitem[Nis84]{Nisnevich}
Yevsey~A. Nisnevich.
\newblock Espaces homog\`enes principaux rationnellement triviaux et
  arithm\'etique des sch\'emas en groupes r\'eductifs sur les anneaux de
  {D}edekind.
\newblock {\em C. R. Acad. Sci. Paris S\'er. I Math.}, 299(1):5--8, 1984.

\bibitem[SdJ10]{dJS8}
Jason Starr and Johan de~Jong.
\newblock Almost proper {GIT}-stacks and discriminant avoidance.
\newblock {\em Doc. Math.}, 15:957--972, 2010.

\bibitem[Ser95]{SerrePP}
Jean-Pierre Serre.
\newblock Cohomologie galoisienne: progr\`es et probl\`emes.
\newblock {\em Ast\'erisque}, (227):Exp.\ No.\ 783, 4, 229--257, 1995.
\newblock S\'eminaire Bourbaki, Vol.\ 1993/94.

\bibitem[Ser02]{GalCoh}
Jean-Pierre Serre.
\newblock {\em Galois cohomology}.
\newblock Springer Monographs in Mathematics. Springer-Verlag, Berlin, english
  edition, 2002.
\newblock Translated from the French by Patrick Ion and revised by the author.

\bibitem[Sta10]{SStrsbg}
Jason~Michael Starr.
\newblock Rational points of rationally simply connected varieties.
\newblock In {\em Vari\'et\'es rationnellement connexes: aspects
  g\'eom\'etriques et arithm\'etiques}, volume~31 of {\em Panor. Synth\`eses},
  pages 155--221. Soc. Math. France, Paris, 2010.

\bibitem[Sta13]{SPAC}
Jason Starr.
\newblock Degenerations of rationally connected varieties and {PAC} fields.
\newblock In {\em A celebration of algebraic geometry}, volume~18 of {\em Clay
  Math. Proc.}, pages 577--589. Amer. Math. Soc., Providence, RI, 2013.

\bibitem[{Sta}17]{stacks-project}
The {Stacks Project Authors}.
\newblock {S}tacks {P}roject.
\newblock \url{http://stacks.math.columbia.edu}, 2017.

\bibitem[Ste65]{Steinberg}
Robert Steinberg.
\newblock Regular elements of semisimple algebraic groups.
\newblock {\em Inst. Hautes \'Etudes Sci. Publ. Math.}, (25):49--80, 1965.

\bibitem[Sus84]{Suslin84}
A.~A. Suslin.
\newblock Algebraic {$K$}-theory and the norm residue homomorphism.
\newblock In {\em Current problems in mathematics, Vol. 25}, Itogi Nauki i
  Tekhniki, pages 115--207. Akad. Nauk SSSR Vsesoyuz. Inst. Nauchn. i Tekhn.
  Inform., Moscow, 1984.

\bibitem[SX]{StarrXu}
Jason~Michael Starr and Chenyang Xu.
\newblock Rational points of rationally simply connected varieties over global
  function fields.

\bibitem[Zhu11]{Zhu2}
Yi~Zhu.
\newblock Fano hypersurfaces in positive characteristic.
\newblock 2011.

\end{thebibliography}
\bibliographystyle{alpha}

\end{document}